%% file: algfill.tex
\documentclass[12pt,reqno]{amsart}
\usepackage[utf8]{inputenc}
\usepackage[numbers]{natbib}
\usepackage{hyperref}
\normalsize
\usepackage{enumerate}
\usepackage{amscd,fullpage,amssymb,amsmath,amsthm}
\usepackage{tikz}
\usepackage[all]{xy}
\usepackage{chngcntr}
\usepackage{longtable}
\usepackage{mathtools}
\usepackage{comment}
\usepackage{rotating}

\usetikzlibrary{patterns}

\theoremstyle{definition}
\newtheorem{sa}{Theorem}[section]
\renewcommand{\thesa}{\arabic{section}.\arabic{sa}}
\newtheorem{de}[sa]{Definition}
\newtheorem{rp}[sa]{Repetition}
\newtheorem{lm}[sa]{Lemma}
\newtheorem{pr}[sa]{Proposition}

\newtheorem{rem}[sa]{Remark}
\newtheorem{ex}[sa]{Example}
\newtheorem{com}[sa]{Comment}

\newtheorem{qu}[sa]{Question}

\newtheorem{con}[sa]{Construction}
\newtheorem{nota}[sa]{Notation}

\newcommand{\Z}{\mathbb{Z}}
\newcommand{\Q}{\mathbb{Q}}
\newcommand{\N}{\mathbb{N}}
\newcommand{\R}{\mathbb{R}}

\newcommand{\Set}{\mathbf{Set}}
\newcommand{\s}{\mathbf{s}}

\newcommand{\Top}{\mathbf{Top}}

\newcommand{\cgda}{\mathbf{cgda}}

\newcommand{\del}{\partial}
\newcommand{\id}{\mathrm{id}}

\newcommand{\Hom}{\mathrm{Hom}}

\newcommand{\width}{\mathrm{width}}
\newcommand{\union}{\cup}
\newcommand{\intersection}{\cap}
\newcommand{\bigunion}{\bigcup}
\newcommand{\bigintersection}{\bigcap}
\newcommand{\ev}{\mathrm{ev}}

\newcommand{\Fill}{\mathrm{Fill}}
\newcommand{\fil}{\mathrm{fill}}

\newcommand{\bigslant}[2]{{\raisebox{.2em}{$#1$}\left/\raisebox{-.2em}{$#2$}\right.}}
\newcommand{\vol}{\operatorname{vol}}
\newcommand{\rk}{\operatorname{rk}}

\newcommand{\defeq}{\vcentcolon=}

\DeclareMathOperator*{\im}{\mathrm{im}}
\DeclareMathOperator*{\inte}{\mathrm{int}}

\numberwithin{equation}{section}

\title{Algebraic filling inequalities and cohomological width}
\author{Meru Alagalingam}
\begin{document}
\maketitle
\begin{abstract}
In his work on singularities, expanders and topology of maps, Gromov showed, using isoperimetric inequalities in graded algebras, that every real valued map on the \mbox{$n$-torus} admits a fibre whose homological size is bounded below by some universal constant depending on $n$. He obtained similar estimates for maps with values in finite dimensional complexes, by a Lusternik--Schnirelmann type argument. 

We describe a new homological filling technique which enables us to derive sharp lower bounds in these theorems in certain situations. This partly realizes a programme envisaged by Gromov. 

In contrast to previous approaches our methods imply similar lower bounds for maps defined on products of higher dimensional spheres.
\end{abstract}
\section{Introduction}
This paper is profoundly inspired by \cite{Gromov2009} and \cite{Gromov2010}, in which, among others, the following two theorems were shown.
\begin{sa}[{{\cite[pp.~424]{Gromov2010}}}]\label{gromovtorus}
Let $k<\frac{n}{2}$ and let $T^n$ denote the $n$-dimensional torus. Every continuous map $f\colon T^n\to\R$ admits a point $y\in\R$ such that the rank of the restriction homomorphism satisfies \[\rk\left[H^k(T^n)\to H^k(f^{-1}(y))\right]\geq\left(1-\frac{2k}{n}\right){n\choose k}\text{.}\]
\end{sa}
The second theorem is the so-called \emph{maximal fibre inequality}:
\begin{sa}[{{\cite[pp.~13]{Gromov2009}}},{{\cite[Section~4.2]{Gromov2010}}}]\label{cortorus}
Let $Y^q$ be a $q$-dimensional simplicial complex and $n\geq p(q+1)$. Every continuous map $f\colon T^n\to Y$ admits a point $y\in Y$ satisfying \[\rk\left[H^*T^n\to H^*(f^{-1}(y))\right]\geq 2^p\text{.}\]
\end{sa}
In the theorems above $H^*$ shall denote \v{C}ech cohomology with coefficients in $\Z$.

\begin{de}[Cohomological width]\label{cowidth}
Let $R$ be a coefficient ring such that the rank of a homomorphism between $R$-modules makes sense, e.g. $\Z$, $\Z_2$ or $\Q$. For every $y\in Y$ we can consider the rank of the \v{C}ech cohomology restriction homomorphism \[H^*(X;R)\to H^*(f^{-1}y;R)\text{.}\]
\begin{enumerate}[(i)]
\item For every continuous map $f\colon X\to Y$ the expressions
\begin{align*}
\width_*(f;R)\defeq &\max_{y\in Y}\rk\left[H^*(X;R)\to H^*(f^{-1}y;R)\right]\text{ and}\\
\width_k(f;R)\defeq &\max_{y\in Y}\rk\left[H^k(X;R)\to H^k(f^{-1}y;R)\right]
\end{align*}
are called the \emph{total} or \emph{degree $k$ cohomological width of $f$}.
\item For fixed topological spaces $X$ and $Y$ the minima
\begin{align*}
\width_*(X/Y;R)\defeq &\min_{f\in C(X,Y)}\width_*(f;R)\text{ and}\\
\width_k(X/Y;R)\defeq &\min_{f\in C(X,Y)}\width_k(f;R)
\end{align*}
where $C(X,Y)$ denotes the set of all continuous maps $f\colon X\to Y$ are called the \emph{total} or \emph{degree $k$ cohomological width of $X$ over $Y$}.
\end{enumerate}
\end{de}
For every continuous map $f\colon X\to Y$ the preimages of points are called the \emph{fibres of $f$} and $\width_k(f)$ gives a lower bound for the topological complexity of one fibre of $f$. The expression $\width_k(X/Y)$ measures the complexity of $X$ in terms of continuous maps to $Y$.

Up to now Theorems \ref{gromovtorus} and \ref{cortorus} have been essentially the only two inequalities about cohomological width. In this paper we will give new lower bounds for $\width_k(X/Y)$ where $X$ and $Y$ are fixed manifolds.

A careful analysis of the proof of Theorem \ref{cortorus} shows that this $y\in Y$ actually satisfies
\begin{align}\label{purely}
\rk\left[H^k(T^n)\to H^k(f^{-1}(y))\right]\geq{p\choose k}
\end{align}
for every $0\leq k\leq p$.

We can compare the different lower bounds, e.g. for $\width_k(T^{2p}/\R)$: Theorem \ref{gromovtorus} yields
\begin{align}\label{ungl1}
\width_k(T^{2p}/\R)\geq\left(1-\frac{k}{p}\right){2p\choose k}
\end{align}
whereas we get from Theorem \ref{cortorus} that
\begin{align}\label{ungl2}
\width_k(T^{2p}/\R)\geq{p\choose k}\text{.}
\end{align}
The bound (\ref{ungl1}) is significantly stronger than (\ref{ungl2}) but the latter holds for all $1$-dimensional target spaces $Y^1$, not just $Y^1=\R$.

When investigating $\width_k(X/Y)$ we will call the dimension of the target space $Y$ the \emph{codimension} of the cohomological width problem. Theorem \ref{gromovtorus} is a codimension $1$ result and its proof uses so-called \emph{isoperimetric inequalities in algebras}. Theorem \ref{cortorus} on the other hand is a result admitting target spaces $Y^q$ of arbitrary codimension $q\geq 1$. Its proof is far less geometric and uses a Lusternik--Schnirelmann type argument. This argument and isoperimetric inequalities in algebras have been the only known techniques to prove cohomological waist inequalities.

Using a certain \emph{filling argument in a space of $(n-q)$-cycles in $T^n$} -- which we will sketch in a moment -- we sharpen estimate (\ref{purely}) as follows.
\begin{sa}\label{estimate}
If $N^q$ is a manifold we have \[\width_1(T^n/N)=n-q\text{,}\] i.e. for every continuous $f\colon T^n\to N$ there exists a point $y\in N$ such that \[\rk\left[H^1(T^n)\to H^1(f^{-1}(y))\right]\geq n-q\text{.}\]
\end{sa}
Any projection $f\colon T^n\to T^q$ shows that this inequality is sharp. A slightly more general construction shows equality can happen for every $q$-dimensional target manifold $N$.

It is the first non-trivial sharp evaluation of cohomological width, slightly improves the best known lower bound for $\width_1(T^n/\R)$ coming from Theorem \ref{gromovtorus} and generalises to arbitrary source manifolds that need not be tori but can be arbitrary essential $m$-manifolds with fundamental group $\Z^n$ (cf. Theorem \ref{essential2}).

Gromov asked (cf. \cite[Section 4.1 and Section 4.13 D]{Gromov2010}) whether one could use minimal models to prove cohomological waist inequalities. Using rational homotopy theory we could indeed prove the following estimate about cartesian powers of higher-dimensional spheres.

\begin{sa}\label{higherdeg}
Let $p\geq 3$ be odd and $n\leq p-2$. Consider $M=(S^p)^n$ or any simply connected, closed manifold of dimension $pn$ with the rational homotopy type $(S^p)_{\Q}^n$. For any orientable manifold $N^q$ we have \[\width_p(M/N;\Q)\geq n-q\text{.}\]
\end{sa}

Consider the map \[f\colon\left(S^p\right)^n=\left(S^p\right)^{n-q}\times\left(S^p\right)^q\longrightarrow\left(S^p\right)^q\stackrel{g^q}{\longrightarrow}\R^q\] where the first map is a projection and the second one is the $q$-fold Cartesian power of a nonconstant map $g\colon S^p\to\R$. All fibres $f^{-1}(y)$ are of the form $\left(S^p\right)^{n-q}\times A^q\subset\left(S^p\right)^n$ for some proper subset $A\subsetneq S^p$. This proves at least the equality \[\width_p\left(M/N;\Q\right)=n-q\] for $M=\left(S^p\right)^n$.

We have not been able to remove the assumption $n \leq p-2$ in this theorem but suspect that this can be done. Theorems \ref{gromovtorus} and \ref{cortorus} can be adapted to show $\width_p\left(\left(S^p\right)^n/\R\right)\geq n-2$ and $\width_p\left(\left(S^p\right)^n/Y^q\right)\geq\frac{n}{q+1}$ but our bound is stronger.

The proofs of Theorems \ref{estimate} and \ref{higherdeg}, which admit target manifolds of arbitrary codimension $q\geq 1$, use a new technique that is inspired by the metric filling argument sketched below. Theorem \ref{higherdeg} is the first lower bound on $\width_p$ with $p>1$ that has been proven using this technique.

Recall the important and classic \emph{waist of the sphere inequality}:
\begin{sa}
Every (for simplicity smooth and generic) $\R^q$-valued map $f$ on the unit $n$-sphere admits a point $y\in\R^q$ such that the $(n-q)$-dimensional Hausdorff volume satisfies \[\vol_{n-q}f^{-1}(y)\geq\vol_{n-q}S^{n-q}\] where $S^{n-q}\subset S^n$ is the $(n-q)$-dimension equator in $S^n$.
\end{sa}
Equality can happen e.g. if $f$ is the restriction of a linear projection $\R^{n+1}\to\R^q$.
\begin{proof}[Proof scheme (cf. {\cite[p.~134]{Gromov1983}}, \cite{Guth2014})]
We proceed by contradiction and assume that there is a smooth generic map $f\colon S^n\to\R^q$ such that every fibre $f^{-1}(y)$ satisfies \[\vol_{n-q}f^{-1}(y)<\vol_{n-q}S^{n-q}\text{.}\] Choose a generic and fine triangulation $\mathcal{T}$ of $\R^q$. Fine means that the preimage of every $k$-simplex of $\mathcal{T}$ has arbitrarily small $(n-q+k)$-volume. For the preimages of the vertices of $\mathcal{T}$ this holds by assumption, for the preimages of higher-dimensional simplices this can be achieved by subdivision. The sum of the preimages of the $q$-simplices represent the fundamental class $[S^n]\in H_n(S^n;\Z)$. But using certain metric filling inequalities for $(n-q+l)$-chains in $S^n$ we can inductively construct a cone of $[S^n]$. This contradicts the non-vanishing of the fundamental class $[S^n]$. Thus the assumption that all fibres $f^{-1}(y)$ can have arbitrarily small $(n-q)$-volume must have been false.
\end{proof}
We prove cohomological width inequalities by feeding this proof scheme with a new \mbox{\emph{cohomological filling inequality}} (cf. Filling Lemma \ref{fill}). This executes a plan that was indicated by Gromov in \cite[Section 4.13 D]{Gromov2010}.

This paper is organised as follows: in Section 2 we will define cohomological restriction kernels and their connection to cohomological width. The core ideas already appeard in \cite[Section 4.1]{Gromov2010} but only as rough sketches and we allow ourselves the captatio benevolentiae and give rigorous statements and proofs. In particular we give a complete proof that the waist functional only increases under uniform limits, which allows us to reduce waist inequalities to the case of generic maps. This is important for our and possible further treatment of the subject. In Section 3 we define the space $cl^{n-q}(M)$ of $(n-q)$-cycles in a manifold $M$ such that every continuous map $f\colon M\to N$ between manifolds induces a non-trivial element in the homology of $cl^{n-q}(M)$. We show cohomological filling inequalities and use all of these ingredients to prove our waist inequalities.

{\em Acknowledgements.} This paper arose from my Augsburg dissertation. I want to thank my advisor Bernhard Hanke, Mikhail Gromov, Larry Guth and the anonymous referee.

\section{Cohomological width and restriction kernels}
In this paper we will give lower bounds of $\width_k(X/Y)$ for various fixed manifolds $X$~and~$Y$. In Section 3 the proofs of these are given for generic maps $f\colon X\to Y$, e.g.~we will find a lower bound of $\width_k(f)$ for all smooth $f$ which intersect some smooth triangulation of $Y$ transversally. In this section we will show that the same lower bound will also hold for all continuous $f$ (cf. Proposition \ref{upper}). In other words it is sufficient to prove waist inequalities just for (in some sense) generic maps. This is motivated by a sentence in \cite[p.~417]{Gromov2010} about a quantity which ``may only \emph{increase} under uniform limits of maps''. The aim of this section is to render this precise. The reader may feel free to skip ahead to section 3 where the core argument (the case of generic $f$) is presented.

Let $H^*$ denote \v{C}ech cohomology with coefficients in a ring $R$. The following important observation motivates the rest of this section. For every continuous map $f\colon X\to Y$ and every $y\in Y$ we have
\begin{align}\label{kernel}
\rk\left[ H^*X\to H^*f^{-1}y\right]=\rk\left[\bigslant{H^*X}{\ker\left[H^*X\to H^*f^{-1}y\right]}\right]\text{.}
\end{align}
Therefore we want to systematically study kernels of restriction homomorphisms $ H^*X\to H^*C$ for closed subsets $C\subseteq X$ and this motivates \cite[Section~4.1, ideal valued measures]{Gromov2010} the following

\begin{de}\label{restrictionkernel}[Cohomological restriction kernel]
Let $f\colon X\to Y$ be a continuous map. The map $\kappa_f\colon\mathfrak{P}(Y)\to\mathcal{I}(H^*X)$ from the power set of $Y$ to the set of graded ideals in $H^*X$ defined by \[\kappa_f(C)\defeq\ker\left[H^*X\to H^*\left(f^{-1}(C)\right)\right]\] is called the \emph{cohomological restriction kernel of $X$}. For the sake of legibility we will denote $\kappa_{\id_X}$ by $\kappa_X$.
\end{de}

\begin{rem}\label{widthideal}
\begin{enumerate}[(i)]
\item For any continuous map $f\colon X\to Y$ the cohomological restriction kernel $\kappa_f$ satisfies $\kappa_f(Y)=0$ (normalisation), $\kappa_f(C_1)\supseteq\kappa_f(C_2)$ for $C_1\subseteq C_2$ (monotonicity) and $\kappa_f(\emptyset)=H^*X$ (fullness).
\item Equation (\ref{kernel}) becomes \[\rk\left[H^*X\to H^*f^{-1}y\right]=\rk\left[\bigslant{H^*X}{\kappa_f(y)}\right]\] and similarly \[\rk\left[H^kX\to H^kf^{-1}y\right]=\rk\left[\bigslant{H^kX}{\kappa_f(y)\intersection H^kX}\right]\text{.}\]
\end{enumerate}
\end{rem}

Using the specific features of \v{C}ech cohomology we will derive the following property of cohomological restriction kernels.

\begin{pr}[Continuity]\label{continuity}
Let $X$ be a compact topological space and $f\colon X\to Y$ continuous. The cohomological restriction kernel $\kappa_f$ satisfies \emph{continuity}, i.e. for any decreasing nested sequence of closed subsets $Y\supseteq V_1\supseteq V_2\supseteq V_3\supseteq\dots$ we have
\begin{align}\label{cont}
\kappa_f\left(\bigintersection_{i=1}^\infty V_i\right)=\bigunion_{i=1}^\infty\kappa_f(V_i)\text{.}
\end{align}
\end{pr}
We will reduce this proposition to the so-called \emph{continuity} of \v{C}ech cohomology. In order to state this property properly we need a little preparation.
\begin{de}A \emph{compact pair} $(X,A)$ is a pair of spaces such that $X$ is compact and $A\subseteq X$ is closed. In particular $A$ itself is compact. Let $Z$ be a topological space. A sequence of pairs $(X_i,A_i)\subseteq (Z,Z)$ ($i\in\N$) together with inclusions $\iota_i^j\colon (X_i,A_i)\hookrightarrow (X_j,A_j)$ whenever $i<j$ is called a \emph{nested sequence of pairs in $Z$} and we denote it by $\left((X_i,A_i)_{i\in\N},\iota_i^j\right)$. For such a nested sequence its \emph{intersection} is the topological pair $(X,A)\subseteq(Z,Z)$ defined by $X\defeq \bigintersection_iX_i$ and $A\defeq \bigintersection_iA_i$.
\end{de}
We will only need the following very weak version of continuity.
\begin{sa}[{{Continuity of \v{C}ech cohomology, \cite[Theorem 2.6]{EilenbergSteenrod1952}}}]\label{cechcont}
Let $(X,A)$ be the intersection of a nested sequence of compact pairs. Let $\iota_i\colon (X,A)\hookrightarrow (X_i,A_i)$ denote the inclusion. Each $u\in\check{H}^q(X,A)$ is of the form $\iota_i^*u_i$ for some $i\in\N$ and some $u_i\in\check{H}^q(X_i,A_i)$.
\end{sa}
\begin{proof}[Proof of Proposition \ref{continuity}]
Let us prove the continuity of $\kappa_f$ assuming we have proven that of $\kappa_X$. The subsets $\left(f^{-1}V_i\right)_{i\in\N}$ form a decreasing nested sequence of closed subsets of $X$ and the continuity of $\kappa_X$ implies

\[\kappa_f\left(\bigintersection_{i=1}^\infty V_i\right)=\kappa_X\left(f^{-1}\bigintersection_{i=1}^\infty V_i\right)=\kappa_X\left(\bigintersection_{i=1}^\infty f^{-1}V_i\right)=\bigunion_{i=1}^\infty\kappa_f\left(f^{-1}V_i\right)=\bigunion_{i=1}^\infty\kappa_f(V_i)\text{.}\]

It remains to prove the continuity of $\kappa_X$. Let $V$ denote the intersection $\bigintersection_{i=1}^\infty V_i$. The only inclusion of (\ref{cont}) not following from monotonicity is \[\kappa_f\left(\bigintersection_{i=1}^\infty V_i\right)\subseteq\bigunion_{i=1}^\infty\kappa_X(V_i)\text{,}\] i.e. given a cohomology class $z\in H^qX$ satisfying $z\vert V=0\in H^q V$ we have to show the existence of an index $i\in\N$ such that $z\vert V_i=0$.

Consider the nested sequence of compact pairs given by $(X,V_i)_{i\in\N}$. The intersection of this nested sequence is precisely $(X,V)$. For every $i\in\N$ naturality of the long exact sequence yields the following commutativ diagram.

\[\begin{xy}
  \xymatrix{
      H^q(X,V_i) \ar[r]\ar[d] & H^qX\ar[r]\ar@{=}[d] &   H^qV_i\ar[d] \\
      H^q(X,V) \ar[r]           & H^qX \ar[r]&  H^qV
  }\end{xy}\]
In the diagram above every arrow is given by restriction. Because the class $z\in H^qX$ satisfies $z\vert V=0\in H^qV$ we can lift $z$ to a class $\widetilde{z}\in H^q(X,V)$. By Theorem \ref{cechcont} there exists an index $i\in\N$ and a class $u_i\in H^q(X,V_i)$ such that $u_i\vert(X,V)=\widetilde{z}$. We get $u_i\vert X=z$ and the top horizontal sequence yields $z\vert V_i=0$.
\end{proof}
\begin{rem}
\begin{enumerate}[(i)]
\item The continuity axiom fails if $X$ is not compact. Let $X\defeq B^2\setminus 0=\left\{x\in\R^2\middle\vert 0<x_1^2+x_2^2\leq 1\right\}$ and $V_i\defeq \left\{x\in X\middle\vert x_1\leq\frac{1}{i}\right\}$. The intersection $\bigintersection_iV_i=\{x\in X\vert x_1\leq 0\}$ is contractible but the generator of $\check{H}^1X$ survives when restricted to any $V_i$.
\item Continuity also fails when the $V_i$ are not required to be closed. Consider $X\defeq [0,1]$ and $V_i\defeq (0,\frac{1}{i}]$. We have $H^0X=H^0V_i=\Z$ but $\bigintersection_i V_i=\emptyset$ hence $H^0\bigintersection_iV_i=0$.
\item Most interestingly continuity fails if one uses singular instead of \v{C}ech cohomology. Consider the closed topologist's sine curve
\[S\defeq\underbrace{\left\{(t,\sin t)\middle\vert t>0\right\}}_{S_+}\union\underbrace{\{0\}\times[-1,1]}_{S_0}\subset\R^2\] and $V_i\defeq\{(x,y)\in S\vert x\leq\frac{1}{i}\}$. The space $S$ and all the $V_i$ have exactly two path-components hence $H_{sing}^0S=H_{sing}^0V_i=\Z^2$ but $\bigintersection_iV_i$ is homeomorphic to an interval hence $H_{sing}^0\bigintersection_iV_i=\Z$. Consider the cohomology class $z\in H_{sing}^0S$ which takes the value $1$ on the path-component $S_+$ and $0$ on $S_0$. We have that $z\vert\bigintersection_iV_i=0$ but $z\vert V_i\neq 0$ for all $i$.
\end{enumerate}
\end{rem}

Let $X$ and $Y$ be topological spaces and $R$ a coefficient ring such that the rank of a homormorphism between $R$-modules makes sense, e.g. $\Z$, $\Z_2$ or $\Q$.

Recall from Definition \ref{cowidth} that for every continuous map $f\colon X\to Y$ the \emph{total} or \emph{degree $k$ cohomological width of $f$} is given by
\begin{align*}
\width_*(f)\defeq &\max_{y\in Y}\rk\left[H^*X\to H^*f^{-1}y\right]\text{ and}\\
\width_k(f)\defeq &\max_{y\in Y}\rk\left[H^kX\to H^kf^{-1}y\right]\text{.}
\end{align*}
They give rise to the \emph{waist functionals} $\width_*$ and $\width_k$ both of which are (not necessarily in any sense continuous) maps $C(X,Y)\to\N_0$ where $C(X,Y)$ is the space of all continuous maps $f\colon X\to Y$.

\begin{pr}[Upper semi-continuity of waists]\label{upper}
Let $X$ and $Y$ be compact and $Y$ metrisable. If the \v{C}ech cohomology algebra $H^*X$ is finite dimensional the waist functionals $\width_*$ and $\width_k\colon C(X,Y)\to\N_0$ are upper semi-continuous with respect to the compact-open topology.
\end{pr}
\begin{proof}
We will just show the upper semi-continuity of $\width_*$. The corresponding statement for $\width_k$ can be proven analogously. Endow $Y$ with an arbitrary metric $d$. The compact-open topology is identical with the metric topology induced from the uniform norm. As $C(X,Y)$ is a metric space semi-continuity is equivalent to sequential semi-continuity. So given a sequence of functions $f_n\colon X\to Y$ uniformly converging to $f$ we need to show that $\width_*(f_n)\geq\alpha$ for every $n$ implies \[\width_*(f)\geq\alpha\text{.}\] Hence for every $n$ there exists a point $y_n\in Y$ such that \[\rk\left[H^*X\to H^*f_n^{-1}y_n\right]=\rk\left[\bigslant{H^*X}{\kappa_{f_n}(y_n)}\right]\geq\alpha\] where $\kappa_{f_n}$ is the cohomological restriction kernel of $f_n$ and we used Corollary \ref{widthideal} (ii).

Since $Y$ is sequentially compact we can pass to a subsequence and assume that the $y_n$ converge to some point $y\in Y$ and that the convergences $f_n\to f$ and $y_n\to y$ are controlled by
\begin{align*}
d(y_n,y)<\frac{1}{8n^2}\\
d(f_n,f)<\frac{1}{8n^2}\text{.}
\end{align*}
We claim the following equality of subsets of $X$.
\begin{align}
\bigintersection_{n>0}\left\{x\in X\middle\vert d(f_n(x),y_n)\leq\frac{1}{4n^2}+\frac{1}{n}\right\}=f^{-1}(y)\label{seteq}
\end{align}
Let us first discuss the inclusion ``$\supseteq$'': For every $x\in X$ with $d(f_n(x),y_n)>\frac{1}{4n^2}+\frac{1}{n}$ for some $n$ the reverse triangle inequality implies
\begin{align*}
d(f(x),y)\geq d(f_n(x),y_n)-d(f_n(x),f(x))-d(y_n,y)\\
>\frac{1}{4n^2}+\frac{1}{n}-\frac{1}{8n^2}-\frac{1}{8n^2}=\frac{1}{n}>0
\end{align*}
Similarly the inclusion ``$\subseteq$'' can be shown as follows: If $x\in X$ satisfies $d(f_n(x),y_n)\leq\frac{1}{4n^2}+\frac{1}{n}$ for every $n$ we can conclude
\begin{align*}
d(f(x),y)\leq d(f(x),f_n(x))+d(f_n(x),y_n)+d(y_n,y)\\
<\frac{1}{8n^2}+\frac{1}{4n^2}+\frac{1}{n}+\frac{1}{8n^2}\to 0
\end{align*} and hence $f(x)=y$. This proves (\ref{seteq}).

Moreover we claim that the sets on the left hand side of (\ref{seteq}) are nested, i.e. we have
\begin{align}\label{setnested}
\left\{x\in X\middle\vert d(f_n(x),y_n)\leq\frac{1}{4n^2}+\frac{1}{n}\right\}\supseteq\left\{x\in X\middle\vert d(f_{n+1}(x),y_{n+1})\leq\frac{1}{4(n+1)^2}+\frac{1}{n+1}\right\}\text{.}
\end{align}
If $x\in X$ is an element of the right hand side we have
\begin{align*}
d(f_n(x),y_n)\leq d(f_n(x),f_{n+1}(x))+d(f_{n+1}(x),y_{n+1})+d(y_{n+1},y_n)\\
\leq\frac{1}{4n^2}+\frac{1}{4(n+1)^2}+\frac{1}{n+1}+\frac{1}{4n^2}\leq\frac{1}{4n^2}+\frac{1}{n}\text{.}
\end{align*}
proving (\ref{setnested}).

The continuity axiom (which holds by Proposition \ref{continuity} since $X$ is compact) implies \[\bigunion_{n>0}\kappa_X\left\{x\in X\middle\vert d(f_n(x),y_n)\leq\frac{1}{4n^2}+\frac{1}{n}\right\}=\kappa_Xf^{-1}(y)=\kappa_f(y)\text{.}\] The left hand side is an increasing sequence of ideals in $H^*X$ and since the latter is finitely generated (as an $R$-module) there exists an $n>0$ such that \[\kappa_X\left\{x\in X\middle\vert d(f_n(x),y_n)>\frac{1}{4n^2}+\frac{1}{n}\right\}=\kappa_f(y)  \text{.}\] Monotonicity yields \[\kappa_{f_n}(y_n)\supseteq\kappa_f(y)\] proving \[\rk\left[\bigslant{H^*X}{\kappa_f(y)}\right]\geq\rk\left[\bigslant{H^*X}{\kappa_{f_n}(y_n)}\right]\geq\alpha\text{.}\qedhere\]
\end{proof}
\begin{rem}\label{explain}
\begin{enumerate}[(i)]
\item The waist functionals fail to be lower semi-continuous. Consider the embedding $g\colon S^2\hookrightarrow D^3$ and the sequence $f_n\colon S^2\hookrightarrow D^3$ shrinking $g$ to a point, e.g. $f_n(x)=g(x)/n$. This sequence uniformly converges to the constant map $f$ with value $0\in D^3$ but $\width_2(f_n)=0$ whereas $\width_2(f)=1$.
\item One question which immediately arises about the definition of cohomological width of a map $f\colon X\to Y$ is why we defined it as \[\width_k(f)=\max_{y\in Y}\rk\left[H^kX\to H^kf^{-1}y\right]\] where we could have equally been interested in \[w_k(f)\defeq \max_{y\in Y}\rk H^kf^{-1}y\text{.}\] However this functional $w_k\colon C(X,Y)\to\N_0$ fails to be upper semi-continuous. Consider the composition \[g\colon S^2\hookrightarrow D^3\to[-1,1]\] where the first map is the standard embedding and the second map is the restriction of a linear projection, e.g. onto the $x$-axis. Again we consider the family $f_n(x)\defeq g(x)/n$ which converges uniformly to $f$, the constant map with value $0\in[0,1]$. All fibers of $f_n$ are points or circles so we have $w_1(f_n)=1$ but the only nonempty fiber of $f$ is $S^2$ hence $w_1(f)=0$.

Nevertheless we clearly have $w_k(f)\geq\width_k(f)$ so any lower bound for $\width_k(f)$ is also one for $w_k(f)$.

\item The proposition above fails if $X$ is noncompact. Consider again $X=B^2\setminus 0=\left\{x\in\R^2\middle\vert 0<x_1^2+x_2^2\leq 1\right\}$ and the compositions \[B^2\setminus 0\stackrel{pr_1}{\longrightarrow}\R\stackrel{f_\varepsilon}{\longrightarrow}\R\] where the first map is the projection onto the first coordinate and
\[f_\varepsilon\colon x\mapsto\begin{cases}0\text{ for }x\leq\varepsilon\\x-\varepsilon\text{ for }x>\varepsilon\end{cases}\] for every $\varepsilon\in\R$. The maps $f_{1/n}\circ pr_1$ converge uniformly to $f_0\circ pr_1$ and satisfy $\width_1(f_{1/n}\circ pr_1)=1$ whereas we have $\width_1(f_0\circ pr_1)=0$.

\item Upper semi-continuity also fails if cohomological width is not defined via \v{C}ech but singular cohomology. Define
\[\width_0^{sing}(f)\defeq\max_{y\in Y}\rk\left[H_{sing}^0X\to H_{sing}^0f^{-1}y\right]\] and consider again the closed topologist's sine curve \[S\defeq\underbrace{\left\{(t,\sin t)\middle\vert t>0\right\}}_{S_+}\union\underbrace{\{0\}\times[-1,1]}_{S_0}\subset\R^2\] together with the compositions \[S\stackrel{pr_1}{\longrightarrow}\R\stackrel{f_\varepsilon}{\longrightarrow}\R\text{.}\] They satisfy \[\width_0^{sing}\left(f_{1/n}\circ pr_1\right)=2\] but $\width_0^{sing}\left(f_0\circ pr_1\right)=1$.

\item The proof of Proposition \ref{upper} still shows the upper semi-continuity of $\width_k$ if only $H^kX$ is finitely generated and that of $\width_*$ if $H^*X$ is finitely generated as an $R$-algebra since all finitely generated graded commutative algebras over Noetherian base rings are Noetherian.

On the other hand there must be some finiteness condition on $H^*X$. For $X\defeq\{0\}\union\left\{\frac{1}{n}\middle\vert n\in\N\right\}\subset\R$ the cohomology group $H^0X$ is not finitely generated. We have $\width_0(f_{1/n})=\infty$ for all $n\in\N$ but their limit satisfies $\width_0(f_0)=1$.
\end{enumerate}
\end{rem}

\section{Filling argument}
Prior to prove cohomological waist inequalities we need some preliminaries. We will deal with various kinds of manifolds such as smooth manifolds, topological manifolds and manifolds with corners (cf. \cite{Lee2003}). If any specifier is missing by a manifold we mean a smooth manifold. An example of a smooth manifold with corners up to codimension $k$ is the standard $k$-simplex $\Delta^k$. Another source of examples will given in the following proposition.
\begin{de}[Smooth, embedded simplices]
Let $N^q$ be a manifold. A \emph{smooth, embedded $k$-simplex $\sigma$ in $N$} is a smooth map $\sigma\colon\Delta^k\to N$ such that there exists an open neighbourhood $\Delta^k\subset U\subset\R^k$ and a smooth extension $\widetilde{\sigma}\colon U\to N$ which is an embedding.
\end{de}
\begin{de}[Stratum transversality]
Let $M^n$ and $N^q$ be manifolds without boundary, $f\colon M\to N$ smooth and $\sigma\colon\Delta^k\to N$ a smooth, embedded simplex. We say that $f$ intersects $\sigma$ \emph{stratum transversally} if $f$ intersects the interior of $\sigma$ transversally and all of its faces stratum transversally. The map $f$ intersects a $0$-simplex stratum transversally iff its image point is a regular value of $f$.
\end{de}
\begin{pr}[Generic preimages of simplices]\label{genpre}
Let $M^n$ and $N^q$ be closed, oriented manifolds, $\sigma\colon\Delta^k\to N$ a smooth, embedded simplex and $f\colon M\to N$ a smooth map intersecting $\sigma$ stratum transversally.

The preimage $f^{-1}\sigma(\Delta^k)$ is an oriented topological $(n-q+k)$-manifold with boundary \[\del f^{-1}\sigma(\Delta^k)=f^{-1}\sigma(\del\Delta^k)\text{.}\]
\end{pr}
\begin{proof}
Theorem 3 in \cite{Nielsen1982} shows that $f^{-1}\sigma(\Delta^k)$ is a smooth manifold with corners up to codimension $k$ hence it is a topological manifold with boundary. Note that most of the technical assumptions are met since $M$ and $N$ do not have boundary. Moreover the theorem states that the codimension $l$ corner points of $f^{-1}\sigma(\Delta^k)$ are precisely the preimages of codimension $l$ corner points of $\Delta^k$, in particular $\del f^{-1}\sigma(\Delta^k)=f^{-1}\sigma(\del\Delta^k)$.
\end{proof}
Mind the following notational convention.
\begin{nota}
In the situation of Proposition \ref{genpre} we frequently denote the preimage of a simplex $\sigma\colon\Delta^k\to N$ by \[F_{\sigma}\defeq f^{-1}\sigma(\Delta^k)\] and similarly \[F_{\del\sigma}\defeq f^{-1}\sigma(\del\Delta^k)=\del F_{\sigma}\text{.}\] We will often use this notation without explicitly mentioning it.
\end{nota}
\begin{de}[Smooth triangulations]
Let $N^q$ be a smooth manifold. A \emph{smooth triangulation $\mathcal{T}=(K,\varphi)$ of $N$} consists of a finite simplicial complex $K$ together with a homeomorphism $\varphi\colon\vert K\vert\to N$ such that the restriction of $\varphi$ to any simplex yields a smooth, embedded simplex in $N$. The set of all of these smooth $k$-simplices of $\mathcal{T}$ shall be denoted by $\mathcal{T}_k$. We will often omit the specification \emph{smooth} and simply talk about \emph{a triangulation} and its \emph{simplices}.

Let $R$ be a coeffcient ring. If $N^q$ is $R$-oriented a triangulation $\mathcal{T}$ is called \emph{$R$-oriented} iff the sum of the elements in $\mathcal{T}_q$, i.e. the top-dimensional  simplices, represents the $R$-oriented fundamental class of $N^q$.
\end{de}
Any smooth manifold $N$ admits a smooth triangulation \cite[Theorem~10.6]{Munkres1967}.
\begin{pr}\label{cyclesmot}
Let $f\colon M^n\to N^q$ be a smooth map between closed $R$-oriented manifolds, $\mathcal{T}$ an $R$-oriented triangulation of $N$ such that $f$ intersects all the simplices $\sigma\in\mathcal{T}_q$ stratum transversally. For $k=0,\ldots,q$ we can inductively assign singular chains $c_{\sigma}\in C_{n-q+k}(F_{\sigma};R)$ to every $\sigma\in\mathcal{T}_k$ such that the following properties hold.
\begin{enumerate}[(i)]
\item For $\sigma\in\mathcal{T}_0$ the chain $c_{\sigma}\in C_{n-q}(F_{\sigma};R)$ represents the (correctly oriented) fundamental class of $F_{\sigma}$.
\item For $1\leq k\leq q$ and $\sigma\in\mathcal{T}_k$ we can view the sum
\begin{align}\label{disclaimer1}
\sum_{i=0}^k(-1)^ic_{\del_i\sigma}
\end{align}
as an element of $C_{n-q+k-1}(\del F_{\sigma};R)$ and this represents the (correctly oriented) fundamental class of $\del F_{\sigma}$ with the boundary orientation. The element $c_{\sigma}\in C_{n-q+k}(F_{\sigma};R)$ satisfies
\begin{align}\label{disclaimer2}
\del c_{\sigma}=\sum_{i=0}^k(-1)^ic_{\del_i\sigma}
\end{align}
as an equation in $C_{n-q+k-1}(F_{\sigma};R)$ and $c_{\sigma}$ represents the (correctly oriented) relative fundamental class in $H_{n-q+k}\left(F_{\sigma},\del F_{\sigma};R\right)$.
\item The sum
\begin{align}\label{disclaimer3}
\sum_{\sigma\in\mathcal{T}_q}c_{\sigma}\in C_{n}(M;R)
\end{align}
represents the (correctly oriented) fundamental class of $M$.
\end{enumerate}
\end{pr}

\begin{center}
\def\svgwidth{0.95\textwidth}
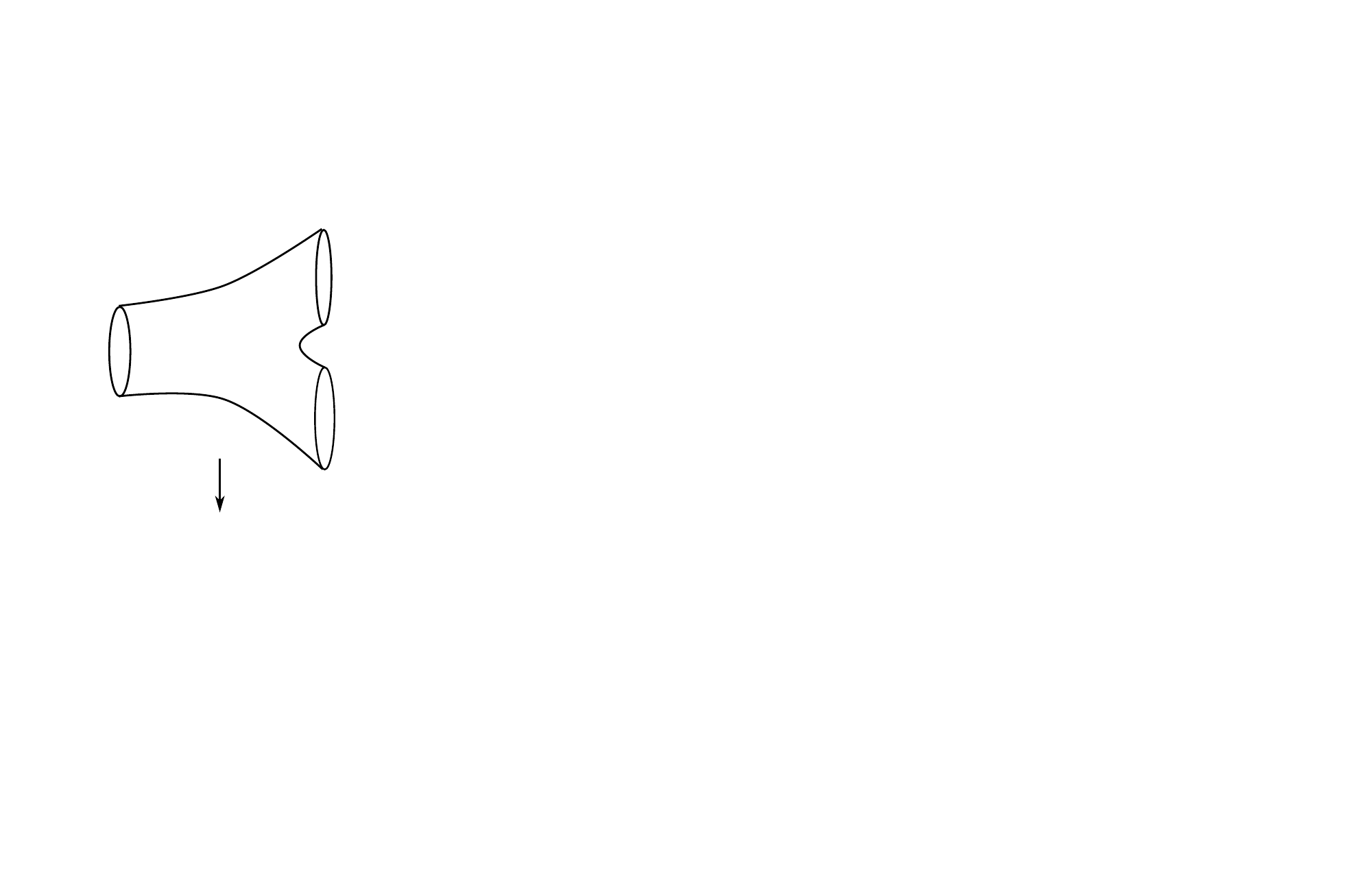
\end{center}

In the example picture on the right hand side $c_{[u,w]}$ is a cylinder and both $c_{[u,v]}$ and $c_{[v,w]}$ are pairs of pants. The chain $c_{[u,v,w]}$ is a solid torus. The bold line is mapped to the barycentre of $[u,v,w]$ and the farther a point in $c_{[u,v,w]}$ is from this core line the closer it is mapped to $\del[u,v,w]$.

\begin{rem}
Technically the summands appearing in the expressions (\ref{disclaimer1}), (\ref{disclaimer2}) and (\ref{disclaimer3}) are elements of different chain groups $C_{n-q+k-1}F_{\del_i\sigma}$ (for varying $i$) or $C_nF_{\sigma}$ (for varying $\sigma$). In order to make sense of the sums and equations we view these summands as chains in the chain group of the larger space $F_{\del\sigma}$ or $M$. For the sake of legibility we omit all the inclusions and their induced maps on chain groups and ask the reader to interpret such equations of cycles in a sensible way. This convention holds for the rest of this paper.
\end{rem}
\begin{proof}[Proof of Proposition \ref{cyclesmot}]
Proposition \ref{genpre} shows that for all $\sigma\in\mathcal{T}_k$ the preimage $F_{\sigma}$ is an oriented topological $(n-q+k)$-manifold with boundary $F_{\del\sigma}$. Hence the notion of fundamental classes makes sense. Bear in mind that both $F_{\sigma}$ and $\del F_{\sigma}$ may be empty or have several components.
\begin{enumerate}[(i)]
\item For every $\sigma\in\mathcal{T}_0$ the preimage $F_{\sigma}$ is a closed oriented $(n-q)$-dimensional submanifold of $M$ and it is easy to arrange (i). We proceed by induction over $k$ and assume that we have constructed chains $c_{\tau}$ for all simplices $\tau\in\mathcal{T}_l$ of dimension $l<k$.
\item A standard calculation shows \[\del\sum_{i=0}^k(-1)^ic_{\del_i\sigma}=\sum_{i=0}^k(-1)^i\del c_{\del_i\sigma}=\sum_{i=0}^k(-1)^i\sum_{j=0}^{k-1}(-1)^jc_{\del_j\del_i\sigma}=0\text{.}\] Hence $\sum_{i=0}^k(-1)^ic_{\del_i\sigma}$ defines a homology class in $H_{n-q+k-1}(\del F_{\sigma})$. For every $0\leq j\leq k$ the induced maps of the inclusions satisfy
\begin{align*}
H_{n-q+k-1}(\del F_{\sigma})&\to H_{n-q+k-1}\left(\del F_{\sigma},\bigunion_{i\neq j}F_{\del_i\sigma}\right)\\
\left[\sum_{i=0}^k(-1)^ic_{\del_i\sigma}\right]&\mapsto\left[(-1)^jc_{\del_j\sigma}\right]\text{.}
\end{align*}

For every $p\in F_{\del_j\sigma}$ the image of these classes in $H_{n-q+k-1}(F_{\sigma},F_{\sigma}\setminus p)$ is the correct local orientation of $F_{\del_j\sigma}$ in the point $p$ where $F_{\del_j\sigma}\subseteq \del F_{\sigma}$ is oriented as the boundary of $F_{\sigma}$. This proves that $\sum_{i=0}^k(-1)^ic_{\del_i\sigma}$ represents the (correctly oriented) fundamental class of $F_{\del\sigma}$.

The fundamental class $[c_{\sigma}]\in H_{n-q+k}(F_{\sigma},\del F_{\sigma})$ satisfies
\begin{align}
\del\colon H_{n-q+k}(F_{\sigma},\del F_{\sigma})&\to H_{n-q+k-1}(\del F_{\sigma})\\
[c_{\sigma}]&\mapsto\left[\sum_{i=0}^k(-1)^ic_{\del_i\sigma}\right]
\end{align}
and the relative cycle $c_{\sigma}$ can be modified so as to achieve equation (\ref{disclaimer2}) on chain level.
\item We have \[\del\sum_{\sigma\in\mathcal{T}_q}c_{\sigma}=\sum_{\sigma\in\mathcal{T}_q}\sum_{i=0}^q(-1)^ic_{\del_i\sigma}=0\] since every $(q-1)$-simplex is the face of exactly two $q$-simplices and inherits different orientations from them. Hence $\sum_{\sigma\in\mathcal{T}_q}c_{\sigma}$ defines a homology class in $H_nM$. Again for every $\tau\in\mathcal{T}_q$ the inclusion $(M,\emptyset)\to\left(M,\bigunion_{\sigma\in\mathcal{T}\setminus\tau}F_{\sigma}\right)$ satisfies
\begin{align*}
H_n(M)&\to H_n\left(M,\bigunion_{\sigma\in\mathcal{T}\setminus\tau}F_{\sigma}\right)\\
\left[\sum_{\sigma\in\mathcal{T}_q}c_{\sigma}\right]&\mapsto c_{\tau}
\end{align*}
and for every $p\in F_{\tau}$ arbitrary the image of these classes in $H_n(M,M\setminus p)$ yields the correct local orientation of $M$ in $p$.\qedhere
\end{enumerate}
\end{proof}
The rest of this section is devoted to the formulation and proof of Proposition \ref{genericfine}, a genericity result which for any map $f\colon M\to N$ guarantees the existence of a triangulation of the target manifold $N$ which is (in a precise sense) generic and fine.
\begin{rp}
Let $M$ and $N$ be manifolds without boundary. We will denote the space of all continuous maps $f\colon M\to N$ by $C^0(M,N)$ and it shall be equipped with the compact-open topology. If $M$ is compact the subspace topology on $C^{\infty}(M,N)\subset C^0(M,N)$ is coarser than the weak $C^{\infty}$-topology.
\end{rp}
\begin{pr}\label{genericfine}
Let $M$ and $N$ be two closed manifolds. For every smooth map $f\colon M\to N$ and every open cover $\mathcal{U}=(U_i)_{i\in I}$ of $N$ there exists a smooth triangulation $\mathcal{T}$ of $N$ and a sequence of smooth maps $f_n\colon M\to N$ uniformly converging to $f$ such that the following properties hold:
\begin{enumerate}[(i)]
\item Every map $f_n$ intersects every simplex $\sigma\in\mathcal{T}_k$ stratum transversally.
\item For every $\sigma\in\mathcal{T}_k$ there exists an index $i\in I$ such that $\sigma(\Delta^k)\subseteq U_i$.
\end{enumerate}
\end{pr}
\begin{proof}
Choose a smooth triangulation $\mathcal{T}=(K,\varphi)$ of $N$ and consider the preimage $\varphi^{-1}\mathcal{U}\defeq (\varphi^{-1}U_i)_{i\in I}$ which is an open cover of $\vert K\vert$. Since $N$ is compact this open cover has a Lebesgue number with respect to some standard metric on $\vert K\vert$. After barycentric subdivision we can assume that every simplex of $\vert K\vert$ is contained in some $f^{-1}U_i$, i.e. its image is contained in $U_i$.

For every smooth, embedded simplex $\sigma\colon\Delta^k\to N$ the subset \[\left\{f\in C^{\infty}(M,N)\middle\vert f\pitchfork\im\sigma\right\}\subseteq C^{\infty}(M,N)\] is a residual in the weak $C^{\infty}$-topology, i.e. it is the countable intersection of open and dense subsets \cite[Transversality Theorem 2.1]{Hirsch1976}. Moreover the Baire category theorem applies to the weak $C^{\infty}$-topology, i.e. every residual set is dense. The set
\begin{align*}
&\left\{g\in C^{\infty}(M,N)\middle\vert\text{every simplex intersects }g\text{ stratum transversally}\right\}\\
=&\bigintersection_{\sigma\text{ simplex of }\mathcal{T}}\left\{g\in C^{\infty}(M,N)\middle\vert\sigma\text{ intersects }g\text{ stratum transversally}\right\}						
\end{align*}
is the countable intersection of residual sets, hence itself residual and therefore dense. Since the compact-open topology is coarser than the weak $C^{\infty}$-topology the claim follows.
\end{proof}

For the rest of this paper $N^q$ always denotes a smooth $q$-manifold. At the beginning we allow $N$ to be disconnected, to have non-empty boundary or to be non-compact. Let us recall Theorem \ref{estimate} which will hold for this general class of target manifolds. We will quickly see that we can restrict ourselves to the case where $N$ is closed and connected.
\begingroup
\def\thesa{\ref{estimate}}
\begin{sa}
Every continuous map $f\colon T^n\to N^q$ admits a point $y\in N^q$ such that the rank of the restriction homomorphism satisfies \[\rk\left[H^1(T^n;\Z)\to H^1(f^{-1}(y);\Z)\right]\geq n-q\text{.}\]
\end{sa}
\addtocounter{sa}{-1}
\endgroup
\begin{rem}\label{sharp}
\begin{enumerate}[(i)]
\item This inequality is non-vacuous only if $n>q$ which we will tacitly assume 
. Furthermore it shows $\width_1(T^n/N)\geq n-q$.
\item Let us assume for the moment that we have proven the theorem for closed connected $N$. We will explain how the theorem extends to manifolds which are possibly disconnected, non-compact or have non-empty boundary. Since $T^n$ is connected we can restrict the target of $f$ to the component which is hit. If $N$ had boundary consider the inclusion $N\hookrightarrow D$ into the double $D$ of $N$. Since $D$ has no boundary we can apply the theorem to the composition \[T^n\stackrel{f}{\longrightarrow}N\hookrightarrow D\] yielding the theorem for $N$.

If $N$ is non-compact we choose a sequence $N_1\subset N_2\subset\ldots\subset N$ such that each $N_i$ is a smooth compact codimension $0$ submanifold with boundary and $\bigunion_{i=1}^{\infty}\inte N_i=N$ (such an exhaustion exists by a strong form of the Whitney embedding theorem where every (even non-compact) manifold can be embedded into some $\R^N$ with closed image). Since $f(T^n)$ is compact it is contained in $N_i$ for some $i\gg 0$, i.e. we can view $f$ as a map $T^n\to N_i$ and we already deduced the theorem for compact manifolds with boundary. For the rest of this paper we will assume the target manifold $N$ to be closed and connected.
\end{enumerate}
\end{rem}
The theorem will essentially follow from the following
\begin{pr}\label{essential}
Let $f\colon T^n\to N^q$ be a smooth map where $N$ is a closed manifold together with a smooth triangulation $\mathcal{T}$ the simplices of which intersect $f$ stratum transversally. Then there exists a simplex $\sigma\in\mathcal{T}_k$ such that the preimage $F_\sigma\defeq f^{-1}\sigma(\Delta^k)$ satisfies \[\rk\left[H^1(T^n;\Z)\to H^1(F_\sigma;\Z)\right]\geq n-q\text{.}\]
\end{pr}
\begin{proof}[Proof of Theorem \ref{estimate} assuming Proposition \ref{essential}]
Assume there is a continuous map $f\colon T^n\to N^q$ such that $\width_1(f)<n-q$. Recall the cohomological restriction kernel $\kappa_f$ from Definition \ref{restrictionkernel} given by
\[\kappa_f(A)\defeq\ker\left[H^*T^n\to H^*f^{-1}A\right]\]
for a subset $A\subseteq N$. Since $T^n$ is compact $\kappa_f$ satisfies the continuity axiom. 

Remark \ref{widthideal} (ii) implies
\[\rk\left[H^1T^n\to H^1f^{-1}y\right]=\rk\left[\bigslant{H^1T^n}{\kappa_f(y)\intersection H^1T^n}\right]\]
for every $y\in N$ and therefore the condition $\width_1(f)<n-q$ translates into
\[\rk\left[\bigslant{H^1T^n}{\kappa_f(y)\intersection H^1T^n}\right]<n-q\text{.}\]
Choose an arbitrary metric on $N$. With respect to this metric we have \[\bigintersection_{m=1}^{\infty}\overline{B\left(y,\frac{1}{m}\right)}=\{y\}\text{.}\] Continuity of $\kappa_f$ yields
\[\bigunion_{m=1}^{\infty}\kappa_f\left(\overline{B\left(y,\frac{1}{m}\right)}\right)=\kappa_f(y)\text{.}\]
Since $H^*T^n$ is finitely generated there exists an $m(y)\gg 0$ depending on $y$ such that \[\kappa_f\left(\overline{B\left(y,\frac{1}{m(y)}\right)}\right)=\kappa_f(y)\text{.}\] For every subset $A\subset f^{-1}B\left(y,\frac{1}{m(y)}\right)$ we have 

\begin{align}
\rk[H^1T^n\to H^1A]\leq\rk\left[H^1T^n\to H^1\left(f^{-1}B\left(y,\frac{1}{m(y)}\right)\right)\right]\\
\leq\rk\left[\bigslant{H^1T^n}{\kappa_f\left(\overline{B\left(y,\frac{1}{m(y)}\right)}\right)\intersection H^1T^n}\right]=\rk\left[\bigslant{H^1T^n}{\kappa_f(y)\intersection H^1T^n}\right]<n-q\text{.}\label{prop2}
\end{align}

Every continuous map $f$ can be uniformly approximated by smooth maps $g_m$. Since $M$ and $N$ are compact and metrisable the upper semi-continuity of $\width_1$ (cf. Proposition \ref{upper}) implies $\width(g_m)\leq\width(f)<n-q$ for $m\gg 0$. So without loss of generality we can assume that $f$ itself is smooth.

Since $N$ is compact we can choose finitely many $y_i\in N$ such that $\left(B\left(y_i,\frac{1}{2m(y_i)}\right)\right)_i$ is an open cover of $N$. Applying Proposition \ref{genericfine} to the smooth map $f\colon T^n\to N^q$ and this finite open cover yields a smooth triangulation $\mathcal{T}$ of $N$ and a sequence of smooth maps $f_m\colon M\to N$ uniformly converging to $f$ such that the following two properties hold:
\begin{enumerate}[(i)]
\item Every map $f_m$ intersects every simplex $\sigma\in\mathcal{T}_k$ stratum transversally.
\item For every simplex $\sigma\in\mathcal{T}_k$ there exists an $i$ (which depends on $\sigma$) such that \begin{align}\sigma(\Delta^k)\subseteq B\left(y_i,\frac{1}{2m(y_i)}\right)\text{.}\label{inclusion}\end{align}
\end{enumerate}
Without loss of generality we can assume that the uniform convergence $f_m\to f$ is controlled by \begin{align}\Vert f_m-f\Vert_{\infty}<\frac{1}{2m}\label{controlled}\text{.}\end{align}

Let $M\defeq\max_i m(y_i)$. For any $\sigma\in\mathcal{T}_k$ we have
\begin{align*}
f_M^{-1}\sigma(\Delta^k)&\subseteq f_M^{-1}B\left(y_i,\frac{1}{2m(y_i)}\right)\text{ for some $i$ by (\ref{inclusion})}\\
&\subseteq f^{-1}B\left(y_i,\frac{1}{m(y_i)}\right)\text{ by (\ref{controlled}).}
\end{align*}
Estimate (\ref{prop2}) shows that $f_M$ contradicts Proposition \ref{essential}.
\end{proof}

\begin{rem}
In the future whenever we want to prove a lower bound for cohomological waist we will reduce it to the proof of a statement similar to Proposition \ref{essential}. We will not carry out this reduction in detail anymore.
\end{rem}
Recall that Theorem \ref{estimate} which we are trying to prove is about a map $f\colon T^n\to N^q$ and its fibres $F_y\defeq f^{-1}y$. Initially we will apply the following statements to the inclusions of the fibres $f_y\colon F_y\hookrightarrow T^n$.
\begin{lm}[Filling Lemma]\label{fill}
Let $K$ be a CW complex, $k\colon K\to T^n$ a continuous map and $\rk H^1(k;\Z)<n-q$. There exists a relative CW complex $(\Fill(k),K)$ and an extension $\fil(k)\colon\Fill(k)\to T^n$ such that the diagram
\[\begin{xy}
  \xymatrix{
  \Fill(k)\ar[rd]^{\fil(k)} & \\
  K\ar[r]_{k}\ar@{^{(}->}[u]^{\iota} & T^n
  }
\end{xy}\]
commutes and the following properties hold.
\begin{enumerate}[(i)]
\item Up to homotopy $\Fill(k)$ is the disjoint sum of a number of tori, one copy for each component of $K$, i.e. \[\Fill(k)\simeq T^{r_1}\amalg T^{r_2}\amalg\ldots\] and the dimensions satisfy $r_i<n-q$. In particular we have $H_{\geq n-q}(\Fill(k);G)=0$ and $H_{\geq n-q}(\iota;G)=0$ for any abelian coefficient group $G$.
\item $(\Fill(k),K)$ is homologically $1$-connected
\item $\rk H^1(\fil(k);\Z)=\rk H^1(k;\Z)$
\end{enumerate}
\end{lm}
Before we prove the lemma we need an analysis of the discrepancy between cohomology and homology.
\begin{rem}\label{discrepancy}
\begin{enumerate}[(i)]
\item For every continuous map $f\colon X\to Y$ we have $\rk H_1(f;\Z)=\rk H^1(f;\Z)$.
\item If $f$ induces an isomorphism on $H_1$ then it induces an isomorphism on $H^1$ (both with coefficients in $\Z$).
\end{enumerate}
\end{rem}
\begin{proof}
This follows from the universal coefficient theorem.\qedhere
\end{proof}
We will frequently change our point of view between cohomology and homology and we will do so without further reference to the remark above.
\begin{nota}\label{grossklein}
From now on we will have to introduce a lot of spaces all of which come with reference maps to $T^n$. As with $f_y\colon F_y\to T^n$ these reference maps are denoted by the lower case letters corresponding to the upper case letters representing the spaces.
\end{nota}
\begin{proof}[Proof of Filling Lemma \ref{fill}]
Let us first discuss the case where $K$ is connected and let $r\defeq \rk H^1(k;\Z)$. By the naturality of the Hurewicz homomorphism the following diagram commutes.
\begin{align}\label{hurewicz}
\begin{xy}
  \xymatrix{\pi_1K\ar[r]\ar@{->>}[d]&\pi_1T^n=\Z^n\ar[d]^{\cong}\\
  H_1(K;\Z)\ar[r]&H_1(T^n;\Z)=\Z^n
  }
\end{xy}
\end{align}
This proves that $\im\pi_1k\subseteq\Z^n$ is also a rank $r$ subgroup. Consider a covering $T^r\times\R^{n-r}\to T^n$ corresponding to this subgroup. There exists a lift $\widetilde{k}\colon K\to T^r\times\R^{n-r}$ such that
\[\begin{xy}
  \xymatrix{
	& T^r\times\R^{n-r}\ar[d]\\
    K \ar[ru]^{\widetilde{k}}\ar[r]_{k} &  T^n}\end{xy}\]
commutes. On the level of fundamental groups this turns into the following diagram.
\[\begin{xy}
  \xymatrix{
	& \pi_1\left(T^r\times\R^{n-r}\right)\ar@{^{(}->}[d]\\
    \pi_1K \ar[ru]^{\pi_1\widetilde{k}}\ar[r]_{\pi_1k} & \pi_1T^n}\end{xy}\] where (by construction of the covering) the vertical arrow is the inclusion $\im\pi_1k\subset\Z^n$. Thus $\pi_1\widetilde{k}$ is obtained from $\pi_1k$ by restricting the target to $\im\pi_1k$, in particular $\pi_1\widetilde{k}$ is surjective. Using the naturality of the Hurewicz homomorphism similar to (\ref{hurewicz}) we conclude that $H_1(\widetilde{k};\Z)$ is surjective.

We want to turn $\widetilde{k}$ into the inclusion of relative CW complex. Substitute $T^r\times\R^{n-r}$ by the mapping cylinder $M_{\widetilde{k}}$ and choose a relative CW approximation $(\Fill(k),K)$ of $(M_{\widetilde{k}},K)$, i.e. there is a weak homotopy equivalence $\Fill(k)\stackrel{\simeq}{\longrightarrow}M_{\widetilde{k}}$ restricting to the identity on $K$. Define $\iota$ and $\fil(k)$ as in the following diagram.
\[\begin{xy}\xymatrix{
\Fill(k)\ar@/^5mm/[rr]^{\fil(k)}\ar[r]^{\simeq} & M_{\widetilde{k}}\ar[r] & T^n\\
 & K\ar@{^{(}->}[lu]^{\iota}\ar[u]\ar[ru]_k &
  }\end{xy}\]
The induced map $H_*(\iota;\Z)$ is an isomorphism for $*=0$ and surjective for $*=1$ since $H_*(\widetilde{k};\Z)$ has these properties. This implies that $(\Fill(k),K)$ is homologically $1$-connected. The surjectivity of $H_1(\iota;\Z)$ also implies $\im H_1(\fil(k);\Z)=\im H_1(k;\Z)$ and together with Remark \ref{discrepancy} (i) we get property (iii). If $K$ is not connected we can apply the construction above to all of its components.
\end{proof}
\begin{rem}\label{expfill}
\begin{enumerate}[(i)]
\item Observe that in the lemma above it is important to assume that the rank $\rk H^1(k;\Z)$ is measured with coefficients in $\Z$. This is due to the usage of the Hurewicz theorem and covering space theory. There is no simple analogue to the Filling Lemma with coefficients in $\Z_2$ since e.g. the double cover map $k\colon S^1\to S^1$ satisfies $\rk H^1(k;\Z_2)=0$ but cannot be filled.
\end{enumerate}
\end{rem}
Actually we could finish the proof of Theorem \ref{estimate} right now but we want to introduce the language of \emph{cycle spaces} which offer a more conceptual viewpoint.

In the following part two kinds of chain complexes will appear, namely singular and the simplicial chain complexes and it should always be clear from the context which one we mean depending on whether we apply it to topological spaces or simplicial sets. Nevertheless in order to avoid confusion we will consistently denote the singular chain complex by $C_*$ and the simplicial chain complex by $C_{\bullet}$.

Let $f\colon M^n\to N^q$ be a smooth map between closed $R$-oriented manifolds, $\sigma$ a smooth embedded $k$-simplex in $N$ which intersects $f$ stratum transversally. Recall Proposition \ref{cyclesmot} by which we can assign to every vertex $v$ of $\sigma$ an $(n-q)$-cycle $c_v$ in $M$ and to any $l$-dimensional face $\tau$ of $\sigma$ an $(n-q+l)$-chain $c_{\tau}$ such that we have \[\del c_{\tau}=\sum_{i=0}^lc_{\del_i\tau}\text{.}\] This motivates the following

\begin{de}[cf. {\cite[Section 2.2]{Gromov2010}}]
Let $(D_*,\del)$ be a chain complex of abelian groups with differential $\del_n\colon D_n\to D_{n-1}$. The \emph{space of $(n-q)$-cycles in $D_*$} is a simplicial set denoted by $cl^{n-q}(D_*,\del)$ the level sets of which are given by \[\left(cl^{n-q}(D_*,\del)\right)_k\defeq \left(cl^{n-q}D_*\right)_k\defeq \Hom\left(C_{\bullet}\Delta[k],D_{*+(n-q)}\right)\text{.}\] Some explanations are in order.
\begin{enumerate}[(i)]
\item $\Delta[k]$ denotes the $k$-dimensional standard simplex in the category $\s\Set$.
\item $C_{\bullet}\Delta[k]$ denotes its normalised chain complex, i.e. the chain groups are generated only by the non-degenerate simplices of $\Delta[k]$.
\item The $\Hom$ set is meant as the set of morphisms of chain complexes of abelian groups.
\end{enumerate}
The right hand side defines a contravariant functor $\Delta\to\Set$ where $\Delta$ is the ordinal number category. This turns $cl^{n-q}D_*$ into a simplicial set. 
\end{de}

The main example of a chain complex $D_*$ to which we want to apply the construction above is the singular chain complex of the source manifold, e.g. a torus.

\begin{rem}\label{bijection}
\begin{enumerate}[(i)]
\item Consider the unique non-degenerate $k$-simplex $c_k\in\Delta[k]_k\subseteq C_k\Delta[k]$. For every $\sigma\in\left(c^{n-q}D_*\right)_k=\Hom\left(C_\bullet\Delta[k],D_{*+(n-q)}\right)$ we will call the image of $c_k$ under $\sigma$ the \emph{top chain of $\sigma$} and denote it by $\ev_k\sigma\in D_{(n-q)+k}$. Sometimes it will be convenient to abbreviate it by $\widehat{\sigma}$.

The maps $\ev_\bullet$ extend and fit together such that \[\ev_\bullet\colon C_\bullet cl^{n-q}D_*\to D_{\bullet+(n-q)}\] is a morphism of chain complexes.

\item We are now able to give a more intuitive description of the simplices of $cl^{n-q}D_*$. The $0$-simplices $c\in\left(cl^{n-q}D_*\right)_0$ are precisely given by their top chains which are $(n-q)$-cycles $c\in\ker\del_{n-q}\subseteq D_{n-q}$. A $1$-simplex $c_{01}\in\left(cl^{n-q}D_*\right)_1$ is given by its two faces $\del_0c_{01}=\vcentcolon c_0\in(cl^{n-1}D_*)_0$ and $\del_1c_{01}=\vcentcolon c_1\in(cl^{n-q}D_*)_0$ and its top chain $\widehat{c_{01}}\in D_{(n-q)+1}$ satisfying

\begin{align}\label{filling}
\del\widehat{c_{01}}=\widehat{c_0}-\widehat{c_1}\text{.}
\end{align}

Thus a $1$-simplex $c_{01}\in\left(cl^{n-q}D_*\right)_1$ consists of two \emph{homologous} chains $\widehat{c_0},\widehat{c_1}\in\ker\del_{n-q}\subseteq D_{n-q}$ \emph{together with a choice of filling $\widehat{c_{01}}\in D_{(n-q)+1}$} satisfying (\ref{filling}). A $2$-simplex $c_{012}\in\left(cl^{n-q}D_*\right)_2$ consists of its three $1$-dimensional faces $c_{12},c_{02},c_{01}\in\left(cl^{n-q}D_*\right)_1$ whose respective $0$-dimensional faces agree and a filling $\widehat{c_{012}}\in D_{(n-q)+2}$ satisfying
\begin{align}\label{filling2}
\del\widehat{c_{012}}=\widehat{c_{12}}-\widehat{c_{02}}+\widehat{c_{01}}\text{.}
\end{align}
In particular for such an $2$-simplex $c_{012}$ to exist the right hand side of (\ref{filling2}) needs to be an $(n-q)+1$-boundary in $D_*$.
\end{enumerate}
\end{rem}

Generally the following lemma tells how $(k+1)$ simplices $\varphi_i\in\left(cl^{n-q}D_*\right)_k$ can be glued together to form the faces of a simplex $\sigma\in\left(cl^{n-q}D_*\right)_{k+1}$ if the obvious homological restriction in $D_*$ vanishes.
\begin{lm}[Gluing Lemma]\label{extension}
Let $\varphi_0,\dots,\varphi_{k+1}\in\left(cl^{n-q}D_*\right)_k$ such that \[\del_i\varphi_j=\del_{j-1}\varphi_i,\quad 0\leq i<j\leq k+1\text{.}\] If there exists an element $\overline{\sigma}\in D_{(n-q)+k+1}$ with \[\del\overline{\sigma}=\sum_{i=0}^{k+1}(-1)^i\widehat{\varphi_i}\] there is a unique $\sigma\in\left(cl^{n-q}D_*\right)_{k+1}$ satisfying $\widehat{\sigma}=\overline{\sigma}$ and $\del_i\sigma=\varphi_i$.
\end{lm}

\begin{proof}A proof is given in \cite[Lemma 4.4.4.]{Alagalingam2016}
\end{proof}

\begin{con}\label{canonical}
Recall Proposition \ref{cyclesmot}. Let $f\colon M^n\to N^q$ be a smooth map between closed $R$-oriented manifolds, $\mathcal{T}$ an $R$-oriented triangulation of $N$ such that $f$ intersects all the simplices $\sigma\in\mathcal{T}_q$ stratum transversally. We can assign to any $\sigma\in\mathcal{T}_k$ a singular chain $c_{\sigma}\in C_{n-q+k}(F_{\sigma};R)$ such that the following properties hold:
\begin{enumerate}[(i)]
\item For $\sigma\in\mathcal{T}_0$ the chain $c_{\sigma}\in C_{n-q}(F_{\sigma};R)$ represents the (correctly oriented) fundamental class of $F_{\sigma}$.
\item For $1\leq k\leq q$ and $\sigma\in\mathcal{T}_k$ we have
\[\del c_{\sigma}=\sum_{i=0}^k(-1)^ic_{\del_i\sigma}\]
as an equation in $C_{n-q+k-1}(F_{\sigma};R)$ and $c_{\sigma}$ represents the (correctly oriented) relative fundamental class in $H_{n-q+k}\left(F_{\sigma},\del F_{\sigma};R\right)$.
\item The sum
\[\sum_{\sigma\in\mathcal{T}_q}c_{\sigma}\in C_{n}(M;R)\]
represents the (correctly oriented) fundamental class of $M$.
\end{enumerate}
For every $\sigma\in\mathcal{T}_0$ we can use Remark \ref{bijection} (ii) to turn the cycles $c_{\sigma}$ into $0$-simplices $z_{\sigma}\in\left(cl^{n-q}C_*(F_{\sigma};R)\right)_0$ satisfying $\widehat{z_{\sigma}}=c_{\sigma}$.

For higher-dimensional $\sigma\in\mathcal{T}_k$ one can inductively use Gluing Lemma \ref{extension} to construct elements $z_{\sigma}\in\left(cl^{n-q}C_*(F_{\sigma};R)\right)_k$ satisfying
\begin{align}
\del_iz_{\sigma}=z_{\del_i\sigma}\quad\text{and}\quad\widehat{z_{\sigma}}=c_{\sigma}\text{.}
\end{align}

The simplicial chain \[Z(f,\mathcal{T})\defeq \sum_{\sigma\in\mathcal{T}_q}z_{\sigma}\] can be viewed as an element in $C_qcl^{n-q}C_*(M;R)$ and it satisfies $\del Z(f,\mathcal{T})=0$. Since \[\ev_{\bullet}\colon C_{\bullet}cl^{n-q}C_*(M;R)\to D_{\bullet+(n-q)}\] is a morphism of chain complexes and maps $Z(f,\mathcal{T})$ to $\sum_{\sigma\in\mathcal{T}_q}c_{\sigma}$ we conclude that $[Z(f,\mathcal{T})]\neq 0$ in $H_qcl^{n-q}C_*(M;R)$.
\end{con}

\begin{com}
\begin{enumerate}[(i)]
\item There is an analytic analogue to the construction above. Let $M^n\subset\R^N$ be a smooth closed embedded manifold, $I_k(M)$ be the topological space of integral currents with the flat topology and $Z_k(M)\subset I_k(M)$ the subspace of cycles. In \cite{FrederickJustinAlmgren1962} Almgren proved that the homotopy groups of the latter are given by \[\pi_iZ_k(M)\cong H_{i+k}(M)\text{.}\] A priori the homotopy groups of a space do not determine its homotopy type since it could have non-zero $k$-invariants but in the case of $Z_k(M)$ the topological group completion theorem implies that the $k$-invariants of every topological abelian monoid vanish. In particular we get
\begin{align}\label{analogue}
Z_k(M)\simeq\prod_{i=0}^{n-k}K(H_{i+k}(M),i)\text{.}
\end{align}
One reasonable corollary from this is $\pi_0Z_k(M)=H_kM$. Another consequence is that
\begin{align}\label{correspondence}
\pi_qZ_{n-q}(M)\cong H_n(M)\cong\Z
\end{align}
with the generator given as follows. Let $f\colon M\subset\R^N\to\R^q$ be a generic projection. For any $y\in\R^q$ the preimage $f^{-1}(q)$ defines an $(n-q)$-dimensional integral cycle and the map
\begin{align*}
\Phi_f\colon\R^q&\to Z_{n-q}(M)\\
y&\mapsto f^{-1}(y)
\end{align*}
is continuous and maps everything outside of $\im f$ to the zero cycle. Hence it determines an element $[\Phi_f]\in\pi_qZ_{n-q}(M)$ which is independent of $f$ and corresponds exactly to the fundamental class under the correspondence (\ref{correspondence}).

\vspace{5mm}
\begin{center}
\def\svgwidth{0.7\textwidth}
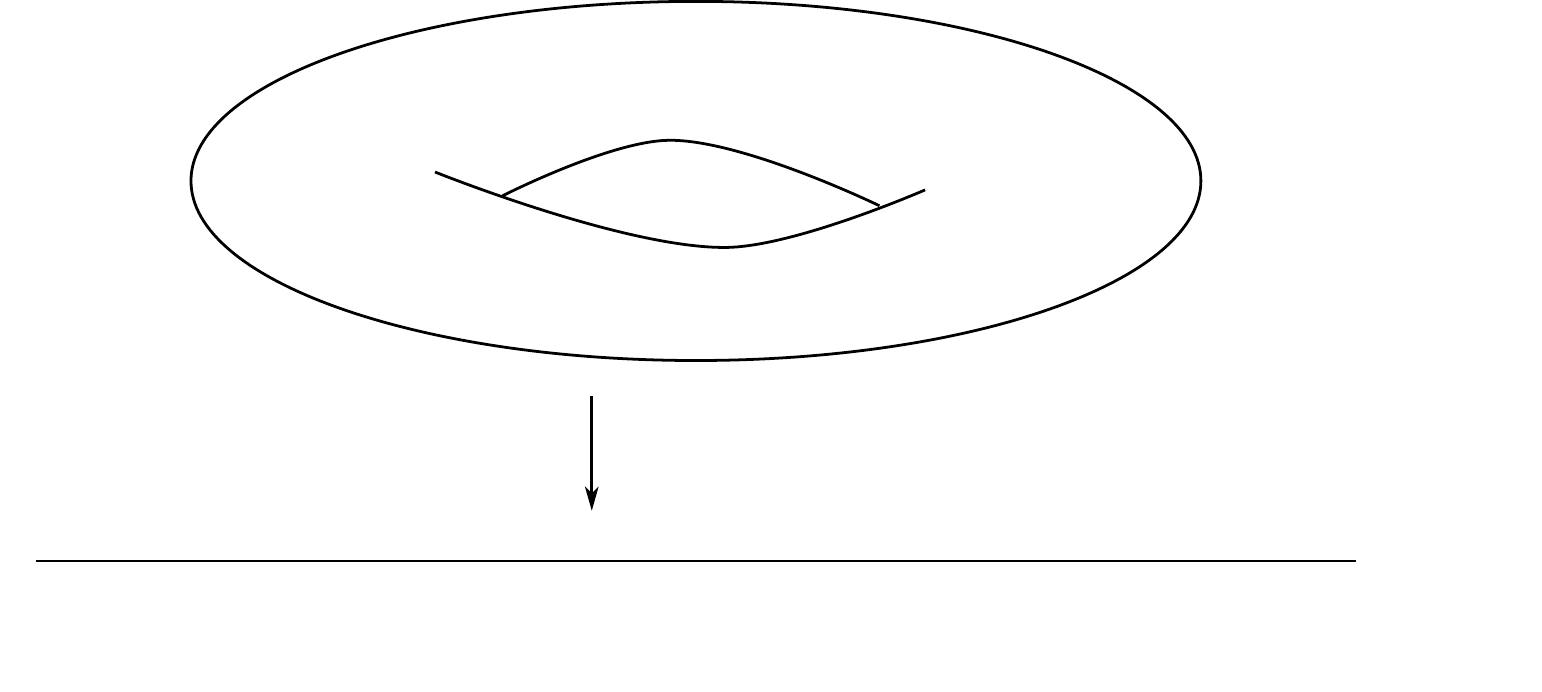
\end{center}

It is an important observation -- especially when proving waist inequalities -- that every map $f\colon M^n\to\R^q$ yields a homotopically nontrivial map $\Phi_f\colon\R^q\to Z_{n-q}(M)$. There are different ways to formalise the notion of spaces of cycles. For obvious reasons we chose a definition with the flavour of algebraic topology.
\item Of course one wonders what is the homotopy type of $cl^{n-q}D_*$ for a given chain complex $D_*$. Up to an index shift $cl^{n-q}$ is just the Dold-Kan correspondence between chain complexes and simplicial abelian groups and from that we get in analogy to (\ref{analogue}) \[cl^{n-q}D_*\simeq\prod_{i=0}^{\infty}K(H_{n-q+i}(D_*),i)\text{.}\]
\item In the construction above the cycle $Z(f,\mathcal{T})$ in $cl^{n-q}C_*(M;R)$ is called the \emph{canonical cycle associated to $f$ and $\mathcal{T}$} and $[Z(f;\mathcal{T})]\in H_qcl^{n-q}C_*(M;R)$ the \emph{canonical homology class}. The cycle $Z(f,\mathcal{T})$ depends heavily on the map $f$ and the triangulation $\mathcal{T}$ whereas one can show that $[Z(f;\mathcal{T})]$ is independent of these choices. We could define the canonical homology class far easier as being represented by the $q$-simplex given by the diagram
\[\begin{xy}
  \xymatrix{
	0\ar[r]\ar[d] & C_q\Delta[q]\ar[r]\ar[d]_{\sigma_q} & \ldots\ar[r] & C_0\Delta[q]\ar[r]\ar[d]^{\sigma_0} & 0\ar[d]\\
	C_{n+1}M\ar[r] & C_{n}M\ar[r] & \ldots\ar[r] & C_{n-q}M\ar[r] & D_{(n-q)-1}\text{.}
	}\end{xy}\]
where $\sigma_q$ maps $c_q$ to a fundamental cycle of $M$ and all other $c_i$ vanish. This cycle arises from the geometric construction above if there exists one large $q$-simplex containing the $\im f$.

However this cycle does not incorporate the map $f$ and the fine triangulation $\mathcal{T}$ in such a way which will enable us to execute the proof of Proposition \ref{essential} which we restate for convenience.
\end{enumerate}
\end{com}
\begingroup
\def\thesa{\ref{essential}}
\begin{pr}
Let $f\colon T^n\to N^q$ be a smooth map where $N$ is a closed manifold together with a smooth triangulation $\mathcal{T}$ the simplices of which intersect $f$ stratum transversally. Then there exists a simplex $\sigma\in\mathcal{T}_k$ such that the inclusion $f_\sigma\colon F_\sigma\defeq f^{-1}\sigma(\Delta^k)\hookrightarrow T^n$ satisfies
\begin{align}\label{interpret}
\rk H^1(f_{\sigma};\Z)\geq n-q\text{.}
\end{align}
\end{pr}
\addtocounter{sa}{-1}
\endgroup
In the following proof there will be a certain unpleasant mixture of coefficients between $\Z$ and $\Z_2$. After all this could not have been totally avoided since we do not want to assume the target manifold $N$ to be orientable which introduces $\Z_2$ coefficients at some places. On the other hand, as explained in Remark \ref{expfill} (iii), the usage of Filling Lemma \ref{fill} forces us to interpret some expressions, e.g. (\ref{interpret}), with coefficients in $\Z$.
\begin{proof}
We proceed by contradiction and assume that such a map $f$ and triangulation $\mathcal{T}$ exist. Recall the simplices $z_{\sigma}\in\left(cl^{n-q}C_*(F_{\sigma};\Z_2)\right)_k$ and the canonical cycle \[Z(f;\mathcal{T})\defeq \sum_{\sigma\in\mathcal{T}_q}z_{\sigma}\in C_qcl^{n-q}C_*(T^n;\Z_2)\] from Construction \ref{canonical}.

We will build the cone of $Z$ inside $cl^{n-q}C_*(T^n;\Z_2)$. For every $\sigma\in\mathcal{T}_k$ we will construct simplices $w_{\sigma}\in\left(cl^{n-q}C_*(T^n;\Z_2)\right)_{k+1}$ satisfying
\begin{align}
\label{zwei}\del_iw_{\sigma}=\begin{cases}w_{\del_i\sigma}&\text{if }0\leq i\leq k\\z_{\sigma}&\text{if }i={k+1}\text{.}\end{cases}
\end{align}
For $\sigma\in\mathcal{T}_0$ and $i=0$ equation (\ref{zwei}) shall be interpreted as $\del_0w_{\sigma}=w_{\del_0\sigma}=0$.

\begin{center}
\def\svgwidth{0.8\textwidth}
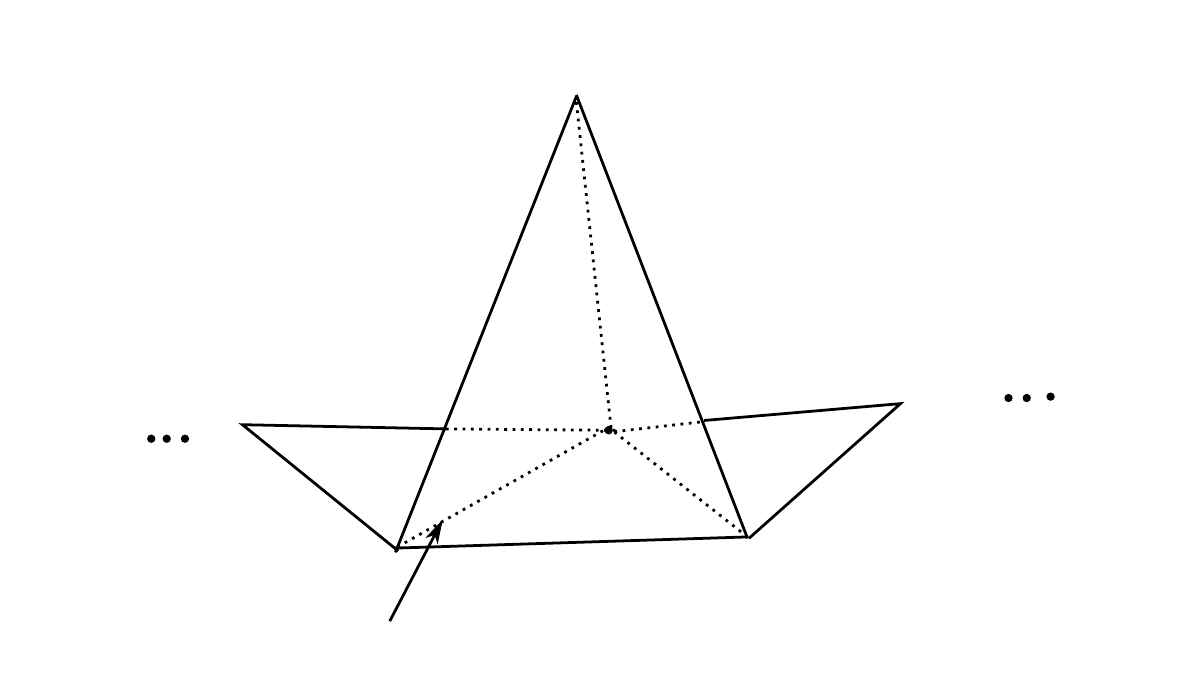
\end{center}

If we constructed such simplices $w_{\sigma}$ the standard cone calculation shows
\[\del\sum_{\sigma\in\mathcal{T}_q}w_{\sigma}=(-1)^{q+1}Z(f;\mathcal{T})\]
contradicting Construction \ref{canonical} where we have seen that $[Z(f;\mathcal{T})]\neq 0$ in $H_qcl^{n-q}C_*(T^n;\Z_2)$. So we are only left with constructing simplices $w_{\sigma}$ satisfying equation (\ref{zwei}).

Recall Notation \ref{grossklein} that every map from a topological space to $T^n$ is denoted by the lower case letter corresponding to the upper case letter representing the space. For every $0\leq k\leq q$ and $\sigma\in\mathcal{T}_k$ we will inductively construct triples $(L_{\sigma},K_{\sigma},F_{\sigma})$ of topological spaces and simplices $w_{\sigma}\in\left(cl^{n-q}C_*(L_{\sigma};\Z_2)\right)_{k+1}$ such that the following properties hold.
\begin{enumerate}[(i)]
\item $(L_{\sigma},F_{\sigma})$ is a homologically $1$-connected relative CW complex and we write
\begin{align}\label{cells}
L_{\sigma}=F_{\sigma}\union e_{\sigma}
\end{align}
where $e_{\sigma}$ is an abbreviation for all the cells which we need to attach to $F_{\sigma}$ in order to obtain $L_{\sigma}$.
\item There are canonical inclusions as in the following diagram.
\[\begin{xy}
  \xymatrix{
	L_{\del_i\sigma}\ar@{^{(}->}[r]\ar@{^{(}->}[rd] & L_{\sigma}\\
	K_{\del_i\sigma}\ar@{^{(}->}[r]\ar@{^{(}->}[u] & K_{\sigma}\ar@{^{(}->}[u]\\
	F_{\del_i\sigma}\ar@{^{(}->}[r]\ar@{^{(}->}[u] & F_{\sigma}\ar@{^{(}->}[u]
  }\end{xy}\]
\item There exist extensions such that the diagram
\[\begin{xy}\label{exte}
  \xymatrix{
  L_{\sigma}\ar[rd]^{l_{\sigma}} & \\
  K_{\sigma}\ar[r]^{k_{\sigma}}\ar@{^{(}->}[u] & T^n\\
	F_{\sigma}\ar[ru]_{f_{\sigma}}\ar@{^{(}->}[u] &
  }\end{xy}\] commutes.
\item $\rk H^1(l_{\sigma};\Z)=\rk H^1(k_{\sigma};\Z)=\rk H^1(f_{\sigma};\Z)<n-q$
\item We have $H_{\geq n-q}(L_\sigma;\Z_2)=0$ and in particular $H_*(K_{\sigma};\Z_2)\to H_*(L_{\sigma};\Z_2)$ vanishes for $*\geq n-q$.
\item The simplices $w_{\sigma}$ satisfy (\ref{zwei}) as a relation of simplices $cl^{n-q}C_*(L_{\sigma};\Z_2)$. Naturally it can also be seen as a relation in $cl^{n-q}C_*(T^n;\Z_2)$.
\end{enumerate}

In the base case $k=0$ we can set $K_{\sigma}\defeq F_{\sigma}$. By assumption we have $\rk H^1(f_{\sigma};\Z)<n-q$ and we can apply Filling Lemma \ref{fill} to it. We get a relative CW complex $(L_{\sigma},F_{\sigma})$ and an extension
\begin{align}
\begin{xy}
  \xymatrix{
  L_{\sigma}\ar[rd]^{l_{\sigma}} & \\
	K_{\sigma}\ar@{^{(}->}[u]\ar[r] &  T^n\\
  F_{\sigma}\ar@{=}[u]^{\id}\ar[ru]_{f_{\sigma}} &
  }\end{xy}
\end{align}
satisfying (iv). Consider the cycle $\widehat{z_{\sigma}}\in C_{n-q}(F_{\sigma};\Z)$ and its image under the inclusion $F_{\sigma}=K_{\sigma}\hookrightarrow L_{\sigma}$. Since $H_{n-q}(L_{\sigma};\Z_2)=0$ there exists a (suggestively denoted) chain $\widehat{w_{\sigma}}\in C_{n-q+1}(L_{\sigma};\Z_2)$ such that
\begin{align}
\del\widehat{w_{\sigma}}=\widehat{z_{\sigma}}\text{.}
\end{align}
Using the Gluing Lemma \ref{extension} we get a simplex $w_{\sigma}\in\left(cl^{n-q}C_*(L_{\sigma};\Z_2)\right)_1$ satisfying (\ref{zwei}) for $k=0$.

Assume $K_{\tau}$, $L_{\tau}$ and $w_{\tau}$ have already been constructed for all simplices $\tau$ of dimension strictly less than $k\geq 1$.

\begin{center}
\def\svgwidth{0.8\textwidth}
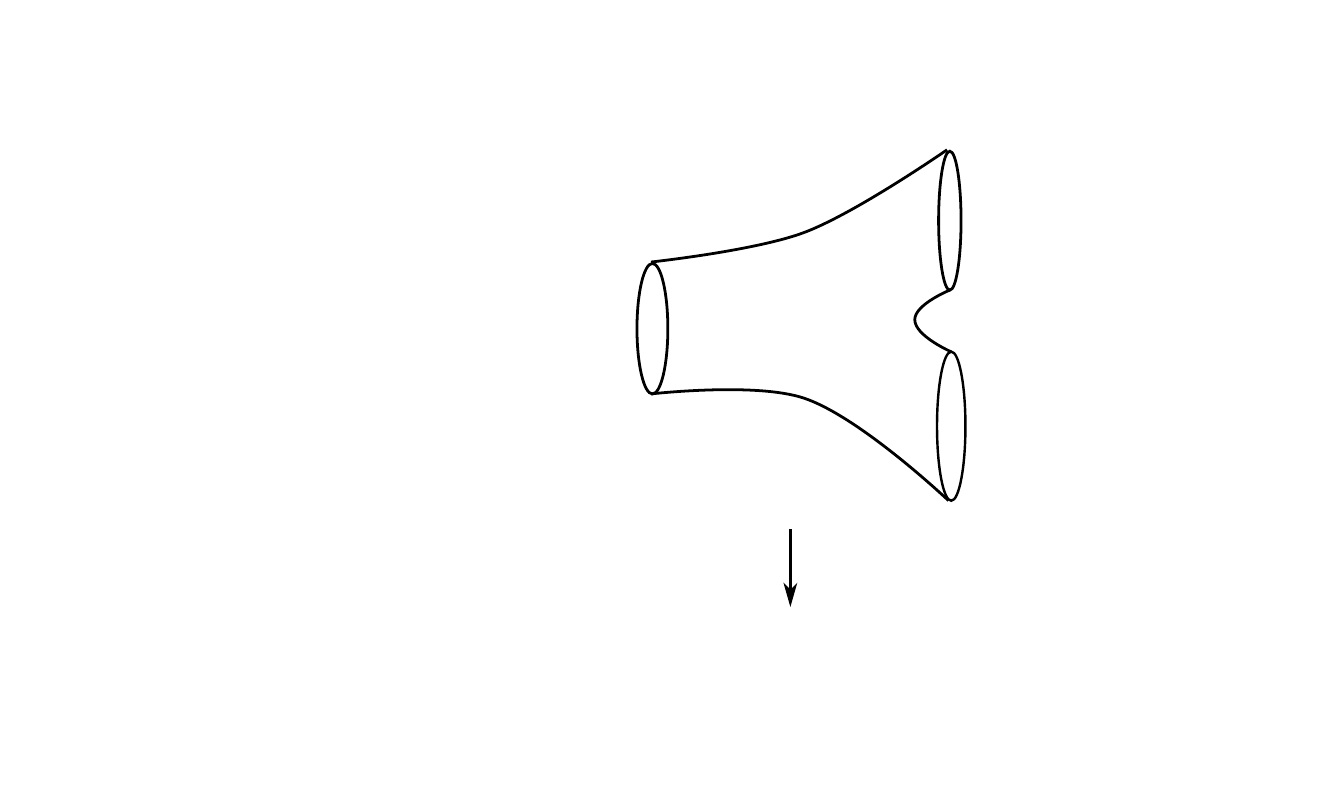
\end{center}

For $\sigma\in\mathcal{T}_k$ and $0\leq i<k$ we inductively define spaces and maps $k_{\sigma}^{(i)}\colon K_{\sigma}^{(i)}\to T^n$ by setting $K_{\sigma}^{(-1)}\defeq F_{\sigma}$, $k_{\sigma}^{(-1)}\defeq f_{\sigma}$ and
\[\begin{xy}
  \xymatrix{
	**[l] K_{\sigma}^{(i-1)}\union\bigunion\limits_{i\text{-dim. faces }\tau\text{ of }\sigma}e_{\tau}=:K_{\sigma}^{(i)}\ar[rd]^{k_{\sigma}^{(i)}} & \\
	K_{\sigma}^{(i-1)}\ar[r]_{k_{\sigma}^{(i-1)}}\ar@{^{(}->}[u] & T^n
  }\end{xy}\]
where we used the notation introduced in equation (\ref{cells}). This is well-defined since the targets of the attaching maps of $e_{\tau}$ are $F_{\tau}$ which canonically are subspaces of $F_{\sigma}\subseteq K_{\sigma}^{(i-1)}$ for every $i$.

We have a homeomorphism
\[\bigvee_{i\text{-dim. faces }\tau\text{ of }\sigma}\left(\bigslant{L_{\del_i\sigma}}{F_{\del_i\sigma}}\right)\stackrel{\cong}{\longrightarrow}\bigslant{K_{\sigma}^{(i)}}{K_{\sigma}^{(i-1)}}\text{.}\]
Since the pairs $(L_{\del_i\sigma},F_{\del_i\sigma})$ are homologically $1$-connected we conclude
\[H_*\left(K_{\sigma}^{(i)},K_{\sigma}^{(i-1)}\right)\cong\bigoplus_{i\text{-dim. faces }\tau\text{ of }\sigma}H_*\left(L_{\del_i\sigma},F_{\del_i\sigma}\right)=0\] for $*=0,1$ proving that the $\left(K_{\sigma}^{(i)},K_{\sigma}^{(i-1)}\right)$ are homologically $1$-connected.

Let $K_{\sigma}\defeq K_{\sigma}^{(k-1)}$, $k_{\sigma}\defeq k_{\sigma}^{(k-1)}$. Since all the $\left(K_{\sigma}^{(i)},K_{\sigma}^{(i-1)}\right)$ are homologically $1$-connected the same holds for $(K_{\sigma},F_{\sigma})$. In particular the inclusion $F_{\sigma}\hookrightarrow K_{\sigma}$ induces a surjective homomorphism $H_1(F_{\sigma};\Z)\to H_1(K_{\sigma};\Z)$. This surjectivity, Remark \ref{discrepancy} (i) and the diagram
\[\begin{xy}
  \xymatrix{		
	K_{\sigma}\ar[r]^{k_{\sigma}} & T^n\\
	F_{\sigma}\ar[ru]_{f_{\sigma}}\ar@{^{(}->}[u] & 
  }\end{xy}\]
show that $\rk H^1(k_{\sigma};\Z)=\rk H_1(k_{\sigma};\Z)=\rk H_1(f_{\sigma};\Z)=\rk H^1(f_{\sigma};\Z)$.

In particular we have $\rk H^1(k_{\sigma};\Z)<n-q$ and we can apply Filling Lemma \ref{fill} to it in order to obtain the space $L_{\sigma}\defeq \Fill(k_{\sigma})$ and the map $l_{\sigma}\defeq \fil(\sigma)$ satisfying (\ref{exte}). The pair $(L_{\sigma},K_{\sigma})$ is homologically $1$-connected and with the same calculation as above we get $\rk H^1(l_{\sigma};\Z)=\rk H^1(k_{\sigma};\Z)$.

Using the inclusions $L_{\del_i\sigma}\subseteq K_{\sigma}$ and $F_{\sigma}\subseteq L_{\sigma}$ we can consider the chain
\begin{align}\label{summe}
y_{\sigma}\defeq \sum_{i=0}^k(-1)^i\widehat{w_{\del_i\sigma}}+(-1)^{k+1}\widehat{z_{\sigma}}\in C_{n-q+k}(K_{\sigma};\Z_2)\text{.}
\end{align}
Since $\del y_{\sigma}=0$ and $H_{n-q+k}(L_{\sigma};\Z_2)=0$ there exists a (suggestively denoted) chain $\widehat{w_{\sigma}}\in C_{n-q+k+1}(L_{\sigma};\Z_2)$ satisfying $\del\widehat{w_{\sigma}}=y_{\sigma}$. Using the Gluing Lemma \ref{extension} we get a simplex $w_{\sigma}\in\left(cl^{n-q}C_*(L_{\sigma};\Z_2)\right)_{k+1}$ satisfying (\ref{zwei}).
\end{proof}
The proof above exhibits a close relationship between $1$-dimensional quantities and fundamental classes and calls to mind the statement and proof of the systolic inequality.

There is a natural generalisation of Theorem \ref{estimate} to essential source manifolds $M$. We will recall this notion.

\begin{de}[Essentialness, cf. \cite{Gromov1983}]Let $G$ be an abelian coefficient group and $M^n$ be a closed connected $G$-oriented manifold with fundamental group $\pi_1(M)=:\pi$ and fundamental class $[M]_G\in H_n(M;G)$. Let $\Phi\colon M\to B\pi$ denote the classifying map of the universal cover $\widetilde{M}\to M$. The manifold $M$ is said to be \emph{$G$-essential} if the image
\begin{align*}
\Phi_*\colon H_n(M;G)&\to H_n(B\pi;G)=H_n(\pi;G)\\
[M]_G&\mapsto \Phi_*[M]_G\neq 0
\end{align*}
does not vanish.
\end{de}
\begin{sa}\label{essential2}
Let $M^m$ be a manifold with fundamental group $\Z^n$ and assume that at least one of the following properties holds:
\begin{enumerate}[(i)]
\item $M$ is $\Z_2$-essential
\item $M$ and $N$ are orientable and $M$ is $\Z$-essential
\end{enumerate}
Then every continuous map $f\colon M\to N$ admits a point $y\in N$ such that the rank of the restriction homomorphism satisfies \[\rk\left[H^1(M;\Z)\to H^1(f^{-1}y;\Z)\right]\geq m-q\text{.}\]
\end{sa}
\begin{rem}
\begin{enumerate}[(i)]
\item With the assumptions of the theorem above we automatically have $m\leq n$ since $H_{>n}(B\Z^n;G)=H_{>n}(T^n;G)=0$. Examples of $G$-essential $n$-manifolds with fundamental group $\Z^n$ ($m=n$) that are not necessarily tori are connected sums of $T^n$ with any simply connected manifold in dimensions $n\geq 3$. If $4\leq m<n$ we can start with a map $\varphi\colon T^m\to T^n$ such that $H_m(\varphi;G)[T^m]\neq 0$ and use surgery to turn this into an essential $m$-manifold with fundamental group $\Z^n$. More generally, manifolds satisfying largeness conditions, such as enlargeability, are $\Z$-essential, compare \cite{Brunnbauer2010}, Theorem 3.6. in connection with Corollary 3.5.
\item For orientable manifolds $M^m$ with fundamental group $\Z^n$ and classifying map $\Phi\colon M\to T^n$ we have the following commutative diagram.
\[\begin{xy}
  \xymatrix{
	H_m(M;\Z)\ar[r]^{H_m(\Phi;\Z)}\ar[d] & H_m(T^n;\Z)\ar[d]\ar[r]^{\cong} & \Z^{n\choose m}\ar[d]\\
	H_m(M;\Z_2)\ar[r]_{H_m(\Phi;\Z_2)} & H_m(T^n;\Z_2)\ar[r]_{\cong} & \Z_2^{n\choose m}
  }\end{xy}\]
	The vertical arrows are change-of-coefficient homomorphisms and the leftmost vertical arrow maps $[M]_{\Z}$ to $[M]_{\Z_2}$. This diagram shows that for manifolds with free abelian fundamental group $\Z_2$-essentialness implies $\Z$-essentialness. This explains the somehow inorganic essentialness assumption in the theorem above.
\end{enumerate}
\end{rem}
\begin{proof}[Proof of Theorem \ref{essential2}]
We only discuss case (i) and indicate the necessary adaptions to the existing proof. Like we reduced Theorem \ref{estimate} to Proposition \ref{essential} we proceed by contradiction and assume that $N$ is connected and closed, there exists a smooth $f\colon M\to N$ and a triangulation $\mathcal{T}$ of $N$ such that the following two properties hold:
\begin{enumerate}[(i)]
\item The smooth simplices of $\mathcal{T}$ intersect $f$ stratum transversally.
\item For every $\sigma\in\mathcal{T}_k$ the inclusion $f_{\sigma}\colon F_{\sigma}\defeq f^{-1}\sigma(\Delta^k)\hookrightarrow M$ satisfies \[\rk H^1(f_{\sigma};\Z)<m-q\text{.}\]
\end{enumerate}
Again for every $\sigma\in\mathcal{T}_k$ we consider the simplices $z_{\sigma}\in\left(cl^{m-q}C_*(F_\sigma;\Z_2)\right)_k$ and the canonical cycle $Z(f;\mathcal{T})\in C_qcl^{m-q}C_*(M;\Z_2)$ from Construction \ref{canonical}. Every $F_\sigma$ comes with a reference map to $M$ and naïvely we would think that we are in need of a replacement for Filling Lemma \ref{fill} where all the maps have target $M$ instead of $T^n$. Instead consider the classifying map $\Phi\colon M\to T^n$. The diagram
\[\begin{xy}
  \xymatrix{
	Z(f;\mathcal{T})\ar@{|->}[rrr]^{\Phi_*}\ar@{|->}[ddd]_{\ev_q} & & & \Phi_*Z(f;\mathcal{T})\ar@{|->}[ddd]_{\ev_q}\\
	 & C_\bullet cl^{m-q}C_*(M;\Z_2)\ar[r]\ar[d] & C_\bullet cl^{m-q}C_*(T^n;\Z_2)\ar[d] & \\
	 & C_{(m-q)+*}(M;\Z_2)\ar[r] & C_{(m-q)+*}(T^n;\Z_2) & \\
	\widehat{Z(f;\mathcal{T})}\ar@{|->}[rrr]_{\Phi_*} & & & \Phi_*\widehat{Z(f;\mathcal{T})}
  }\end{xy}\]
commutes, the bottom left cycle represents the fundamental class $[M]_{\Z_2}\in H_m(M;\Z_2)$ and since $M$ is $\Z_2$-essential the bottom right cycle defines a non-zero element in $H_m(T^n;\Z_2)$. Therefore the top right cycle defines a non-zero element in $H_qcl^{m-q}C_*(T^n;\Z_2)$.

The map $\Phi$ induces an isomorphism on $\pi_1$ as well as on $H_1$ by the Hurewicz theorem and $H^1$ by Remark \ref{discrepancy} (ii) (both with coefficients in $\Z$). This proves that every map $k\colon K\to T^n$ satisfying $\rk H^1(k;\Z)<n-q$ also satisfies \[\rk H^1(\Phi\circ k;\Z)=\rk H^1(k;\Z)<n-q\text{.}\] Hence we can proceed as earlier and deduce a contradiction by constructing a cone of $\Phi_*Z(f;\mathcal{T})$ in $cl^{m-q}C_*(T^n;\Z_2)$ via simplices $w_{\sigma}\in\left(cl^{n-q}C_*(T^n;\Z_2)\right)_{q+1}$ satisfying \[\del_iw_{\sigma}=\begin{cases}w_{\del_i\sigma}&\text{if }0\leq i\leq k\\z_{\sigma}&\text{if }i={k+1}\text{.}\end{cases}\] For $\sigma\in\mathcal{T}_0$ and $i=0$ the equation above shall be interpreted as $\del_0w_{\sigma}=w_{\del_0\sigma}=0$.
\end{proof}
\begin{qu}
\begin{enumerate}[(i)]
\item Theorem \ref{estimate}, the more general Theorem \ref{essential2} and the core input of both, Filling Lemma \ref{fill}, give the impression that we have not proven something about tori but about the geometry of the group $\Z^n$. Are there analogues for other groups $G$? Even in the case where $G$ is abelian with torsion, this is harder because $B\Z_p$ has cohomology classes in arbitrary high degrees and not every cohomology class in $H^*G$ is a product of degree $1$ classes.
\item Michał Marcinkowski has asked whether Theorem \ref{essential2} fails if $M$ has fundamental group $\Z^n$ but is inessential.
\end{enumerate}
\end{qu}

There is another natural generalisation of Theorem \ref{estimate} from tori to cartesian powers of higher-dimensional spheres. Our previous proof of Filling Lemma \ref{fill} used covering space theory and cannot be generalised to simply connected manifolds. Instead we will use rational homotopy theory.
\begingroup
\def\thesa{\ref{higherdeg}}
\begin{sa}
Let $p\geq 3$ be odd and $n\leq p-2$. Consider $M=(S^p)^n$ or any simply connected, closed manifold of dimension $pn$ with the rational homotopy type $(S^p)_{\Q}^n$ and $N^q$ an arbitrary orientable $q$-manifold. Every continuous map $f\colon M\to N$ admits a point $y\in N$ such that the rank of the restriction homomorphism satisfies \[\rk\left[H^p(M;\Q)\to H^p(f^{-1}y;\Q)\right]\geq n-q\text{.}\]
\end{sa}
\addtocounter{sa}{-1}
\endgroup
\begin{rem}
Examples of manifolds $M$ as above that are not $(S^p)^n$ are products of rational homology spheres of dimension $p$ or connected sums of $(S^p)^n$ with rational homology spheres of dimension $pn$.
\end{rem}
In this section the coefficient ring is always $R=\Q$. We assume that the reader already got a rough idea of rational homotopy theory but before we prove the theorem above we will shortly recap the notions and concepts we are going to need (cf. \cite{Felix2005} and \cite{Felix2008}).
\begin{de}[Rationalisations]\label{rational}
For a map $f\colon X\to Y$ between simply connected spaces the following three conditions are equivalent:
\begin{enumerate}[(i)]
\item $\pi_*f\otimes\Q\colon\pi_*X\otimes\Q\to\pi_*Y\otimes\Q$ are isomorphisms
\item $H_*(f;\Q)$ are isomorphisms
\item $H^*(f;\Q)$ are isomorphisms
\end{enumerate}
In this case $f$ is called a \emph{rational homotopy equivalence} which is denoted by
\[\begin{xy}
  \xymatrix{X\ar[r]_{\cong_\Q}^f & Y\text{.}
  }\end{xy}\]
A space $X$ is called \emph{rational} if it is simply connected and all $\pi_*X$ are rational $\Q$-vector spaces. A rational homotopy equivalence between rational spaces is a homotopy equivalence.

For any simply connected $X$ there exists a rational space $X_{\Q}$ and a continuous map $r_X\colon X\to X_{\Q}$ which is a rational homotopy equivalence. The space $X_\Q$ is called the \emph{rationalisation of $X$} and $r_X$ the \emph{rationalisation map of $X$}. With these properties the homotopy type of $X_\Q$ is uniquely determined and is called the \emph{rational homotopy type of $X$}.
\end{de}
\begin{de}[Piecewise polynomial differential forms]
To any topological space $X$ we can associate a \emph{commutative} differential graded algebra (henceforth abbreviated by \emph{cgda}) $A_{PL}(X)\defeq A_{PL}(X;\Q)$. This cgda is called the algebra of \emph{piecewise polynomial differential forms on $X$} and by definition an element $\omega\in A_{PL}^k(X)$ assigns to every singular $n$-simplex in $X$ a \emph{polynomial} degree $k$ differential form on the standard $n$-simplex, consistent with face and degeneracy maps. This yields a contravariant functor $A_{PL}\colon\s\Set\to\cgda$ and there is a natural isomorphism
\begin{align}\label{natural}
H^*A_{PL}(X)\cong H^*(X;\Q)\text{.}
\end{align}
\end{de}
\begin{de}[Sullivan and minimal algebras, minimal models]
A \emph{Sullivan algebra} is a cdga $\left(\bigwedge V,d\right)$ whose underlying algebra is free commutative for some graded $\Q$-vector space $V=\bigoplus_{n\geq 1}V^n$ and such that $V$ admits a basis $(x_{\alpha})$ indexed by a well-ordered set such that $dx_{\alpha}\in\bigwedge(x_{\beta})_{\beta<\alpha}$. It is called a \emph{minimal algebra} if it satisfies the additional property $d(V)\subseteq\bigwedge^{\geq 2}V$.

A morphism of cgdas is called a \emph{quasi-isomorphism} if it induces isomorphisms on all cohomology groups. A quasi-isomorphism \[\left(\bigwedge V,d\right)\to (A,d)\] from a minimal algebra to an arbitrary cgda $(A,d)$ is called a \emph{minimal model} of $(A,d)$. If $X$ is a topological space any minimal model \[\left(\bigwedge V,d\right)\to A_{PL}(X)\] is called a \emph{minimal model of $X$}.
\end{de}
Every simply connected space admits such a minimal model. For any simply connected $X$ the maps $H^*(r_X;\Q)$ are isomorphisms and using (\ref{natural}) we conclude that $A_{PL}(r_X)$ is a quasi-isomorphism. If $m\colon\left(\bigwedge V,d\right)\to A_{PL}X_\Q$ is a minimal model of $X_\Q$ the composition \[\left(\bigwedge V,d\right)\stackrel{m}{\longrightarrow}A_{PL}X_\Q\stackrel{A_{PL}(r_X)}{\longrightarrow}X\] yields a minimal model for $X$.
\begin{ex}[Minimal models of spheres, products]\label{spheres}
\begin{enumerate}[(i)]
\item For the spheres $S_\Q^p$ we can give explicit models depending on the parity of $p$. If $p$ is odd one particular model is given by \[\left(\bigwedge[x],0\right)\to A_{PL}S_\Q^p\] with $\deg x=p$ and $d=0$. If $p$ is even there is a model \[\left(\bigwedge[x,y],d\right)\to A_{PL}S_\Q^p\] with $\deg x=p$, $\deg y=2p-1$, $dx=0$ and $dy=x^2$.
\item If $\left(\bigwedge V,d\right)\to A_{PL}X$ is a minimal model for $X$ and $\left(\bigwedge W,d\right)\to A_{PL}Y$ one for $Y$ then \[\left(\bigwedge[V\oplus W],d\right)\cong\left(\bigwedge V,d\right)\otimes\left(\bigwedge W,d\right)\] is a minimal model for the product $X\times Y$.
\end{enumerate}
\end{ex}
\begin{de}[Spatial realisation]\label{spatial}
There is another contravariant functor $\vert\cdot\vert\colon\cgda\to\Top$, called \emph{spatial realisation}, and for every space $X$ a continuous map \[h_X\colon X\to\vert A_{PL}(X)\vert\text{.}\] These map are called \emph{unit maps} and they are natural in $X$, i.e. for any continuous map $f\colon X\to Y$ the square
\[\begin{xy}
  \xymatrix{
	X\ar[r]^f\ar[d]_{h_X} & Y\ar[d]^{h_Y}\\
	\vert A_{PL}X\vert\ar[r]_{\vert A_{PL}f\vert} & A_{PL}Y
  }\end{xy}\]
commutes.
\end{de}
\begin{sa}\label{unitcomp}
The unit maps $h_X$ are always \emph{rational homology equivalences}, i.e. $H_*(h_X;\Q)$ (or equivalently $H^*(h_X;\Q)$) are isomorphisms. For any rational space $X_\Q$ and any minimal model $m\colon\left(\bigwedge V,d\right)\to A_{PL}X_\Q$ the maps \[h_{X_\Q}\colon X_\Q\stackrel{\simeq}{\longrightarrow}\vert A_{PL}X_\Q\vert\] and \[\vert m\vert\colon\vert A_{PL}X_\Q\vert\stackrel{\simeq}{\longrightarrow}\left\vert\bigwedge V,d\right\vert\] are homotopy equivalences.
\end{sa}
Now we can start proving Theorem \ref{higherdeg}.
\begin{lm}\label{connectedp}
Let $\left(\bigwedge[x_1,\ldots,x_n],0\right)$ be the minimal cgda with all generators concentrated in degree $p$ and $(A,d)$ an arbitrary cgda. Let \[f\colon\left(\bigwedge[x_1,\ldots,x_n],0\right)\to (A,d)\] be a morphism of cdgas such that the induced map on degree $p$ cohomology \[f_*\colon H^p\left(\bigwedge[x_1,\dots,x_n],0\right)\to H^p(A,d)\] has rank less than $n-q$. Then the induced map $f_*$ vanishes in all degrees greater or equal than $(n-q)p$.
\end{lm}
\begin{proof}
The statement is non-vacuous only in degrees divisible by $p$, i.e. $lp$ with $l\geq n-q$. It suffices to prove the case $l=n-q$. Consider an arbitrary monomial of length $n-q$, without loss of generality $x_1\cdot\dots\cdot x_{n-q}$. Due to the rank assumption there exists one factor, without loss of generality $x_{n-q}$, such that $\left[f x_{n-q}\right]$ can be expressed as \[\left[f x_{n-q}\right]=\sum_{i<n-q}\lambda_i\left[f x_i\right]\] and hence

\begin{align*}
\left[f(x_1\cdot\ldots\cdot x_{n-q})\right]=&\left[f x_1\right]\cdot\ldots\cdot \left[f x_{n-q}\right]\\
=&\left[f x_1\right]\cdot\ldots\cdot \left[f x_{n-q-1}\right]\cdot\sum_{i<n-q}\lambda_i\left[f x_i\right]=0\text{.\qedhere}\\
\end{align*}
\end{proof}
\begin{lm}\label{diophantine}
Let $n\leq p-2$. For any $0\leq a<n$ the linear diophantine equation
\begin{align}\label{diop}
\lambda(p-1)+\mu p=np-a
\end{align}
has exactly one solution $(\lambda,\mu)\in\Z_{\geq 0}^2$ given by $(\lambda,\mu)=(a,n-a)$.
\end{lm}
\begin{proof}
The integer solutions of (\ref{diop}) are parametrised by \[\left\{(\lambda,\mu)=(a+kp,(n-a)-k(p-1))\middle\vert k\in\Z\right\}\text{.}\] Then the additional requirement $\lambda,\mu\geq 0$ translates into
\begin{align}\label{diop2}
-\frac{a}{p}\leq k\leq\frac{n-a}{p-1}\text{.}
\end{align}
Since \[\frac{n-a}{p-1}-\left(-\frac{a}{p}\right)\leq\frac{n-a}{p-1}+\frac{a}{p-1}=\frac{n}{p-1}<1\] inequality (\ref{diop2}) has at most one solution. It is easy to check that $(a,n-a)$ satisfies all desired properties.
\end{proof}
The lemma above will enable us to prove the following rational version of Filling Lemma~\ref{fill}.
\begin{lm}[Rational Filling Lemma]\label{ratfill}
Let $p\geq 3$ be odd, $n\leq p-2$, $q<n$ and $k\colon K\to (S^p)_{\Q}^n$ a continuous map with $H^p(k;\Q)<n-q$. There exists a relative CW complex $(\Fill(k),K)$ and an extension $\fil(k)\colon\Fill(k)\to(S^p)_{\Q}^n$ such that the diagram
\[\begin{xy}\xymatrix{
	\Fill(k)\ar[rd]^{\fil(k)} & \\
	K\ar@{^{(}->}[u]^{\iota}\ar[r]_k & (S^p)_{\Q}^n
  }\end{xy}\]
commutes and the following properties hold.
\begin{enumerate}[(i)]
\item $H_{\geq np-q}(\iota;\Q)=0$
\item $H^p(\Fill(k),K;\Q)=0$
\item $\rk H^p(\fil(k);\Q)=\rk H^p(k;\Q)<n-q$
\end{enumerate}
\end{lm}
\begin{proof}
The proof strategy is to solve the problem on the algebraic level of cgdas and then use spatial realisation to obtain the desired spaces and maps. Since $p$ is odd we have a minimal model \[\left(\bigwedge[x_1,\ldots,x_n],0\right)\to A_{PL}(S^p)_{\Q}^n\] with generators $x_i$ concentrated in degree $p$. Consider $k^{\sharp}$ given by the following diagram.
\begin{align}\label{diag2}
\begin{xy}
  \xymatrix{
	A_{PL}K & A_{PL}(S^p)_{\Q}^n\ar[l]_{A_{PL}k}\\
	 & \left(\bigwedge[x_1,\ldots,x_n],0\right)\ar[lu]^{k^{\sharp}}\ar[u]
  }\end{xy}
\end{align}
Morphisms between cdgas are denoted with a lower case letter endowed with the superindex~$\sharp$. This notation shall hint at which continuous map we will get after spatial realisation. The map $k^\sharp$ can be factored as follows.
\begin{align}\label{diag}
\begin{xy}
  \xymatrix{
      A_{PL}K & \\
      \left(\bigwedge\left[H^{p-1}K\oplus\im H^p(k^{\sharp})\right],0\right)\ar[u]^{\iota^{\sharp}} & (\bigwedge[x_1,\ldots,x_n],0)\ar[lu]_{k^{\sharp}}\ar[l]^/-5mm/{g^{\sharp}}
			}\end{xy}
\end{align}
The morphism $g^\sharp$ is the obvious one. The map $\iota^\sharp$ can be defined by \emph{choosing representing cocycles}, i.e. choose $y_i\in A_{PL}^{p-1}K$ such that $[y_i]$ constitutes a basis of $H^{p-1}(A_{PL}K)$ and define
\begin{align*}
\iota^\sharp\colon\left(\bigwedge\left[H^{p-1}K\oplus\im H^p(k^{\sharp})\right],0\right)&\to A_{PL}K\\
[y_i]&\mapsto y_i\\
H^p(k^\sharp)[x_i]&\mapsto k^\sharp x_i\text{.}
\end{align*}
With this definition $H^{p-1}\iota^\sharp$ is surjective and $H^p\iota^\sharp$ is injective both of which will in due course imply (ii) and (iii).

We are left to prove (i) which is equivalent to $H^{\geq pn-q}\iota^\sharp=0$. For $0\leq a\leq q<n$ consider a degree $pn-a$ element $x\in\bigwedge\left[H^{p-1}K\oplus\im H^p(k^{\sharp})\right]$. We will show that $H^{pn-a}\iota^\sharp[x]=0\in H^pA_{PL}K$. Without loss of generality $x$ is a product of $\lambda$ generators of degree $(p-1)$ and $\mu$ generators of degree $p$. Since $n\leq p-2$ Lemma \ref{diophantine} yields $(\lambda,\mu)=(a,n-a)$. Thus $x$ contains at least $n-q$ generators of degree $p$, i.e. \[x=yz_1\cdot\ldots\cdot z_{n-q}\] and the $z_i$ can be written as $z_i=g^\sharp w_i$ for some $w_i\in[x_1,\ldots,x_n]$. We conclude
\begin{align*}
H^{pn-a}\iota^\sharp[x]=[\iota^\sharp x]=[\iota^\sharp(yz_1\cdot\ldots\cdot z_{n-q})]=[(\iota^\sharp y)(\iota^\sharp z_1)\cdot\ldots\cdot(\iota^\sharp z_{n-q})]\\
=[(\iota^\sharp y)(\iota^\sharp g^\sharp w_1)\cdot\ldots\cdot(\iota^\sharp g^\sharp w_{n-q})]=[\iota^\sharp y]H^{(n-q)p}k^\sharp[w_1\cdot\ldots w_{n-q}]\text{.}
\end{align*}
Using the natural isomorphism $H^*A_{PL}(X)\cong H^*(X;\Q)$ we get that $\rk H^p(k^\sharp)<n-q$. Hence we can apply the Lemma \ref{connectedp} to conclude $H^{\geq (n-q)p}(k^\sharp)=0$ proving $H^{pn-a}\iota^\sharp[x]=0$.

Let \[\left(\bigwedge W,0\right)\defeq \left(\bigwedge\left[H^{p-1}K\oplus\im H^p(k^{\sharp})\right],0\right)\text{.}\] After spatial realisation of diagrams (\ref{diag2}) and (\ref{diag}) and introducing the unit maps from Repetition \ref{unitcomp} we get the following diagram.
\[\begin{xy}
  \xymatrix{K\ar[r]^k\ar[d]_{h_K} & (S^p)_{\Q}^n\ar[d]_{h_{(S^p)_{\Q}^n}}^{\simeq}\\
	\vert A_{PL}K\vert\ar[r]\ar[d]_{\vert\iota^\sharp\vert} & \left\vert A_{PL}(S^p)_{\Q}^n\right\vert\ar[d]^{\simeq}\\
	\vert\bigwedge W,0\vert\ar[r]_/-3mm/{\vert g^\sharp\vert} & \vert\bigwedge[x_1,\ldots,x_n],0\vert
	}\end{xy}\]
In this diagram the lower square commutes strictly but the upper one only up to homotopy. The map $h_K$ is a rational cohomology equivalence, in particular we still have that $H^{p-1}(\vert\iota^\sharp\vert\circ h_K)$ is surjective and $H^p(\vert\iota^\sharp\vert\circ h_K)$ is injective. The same theorem states that the right hand side vertical arrows are homotopy equivalences. After choosing homotopy inverses we get the triangle
\[\begin{xy}
  \xymatrix{K\ar[r]^k\ar[d]_{\vert\iota^\sharp\vert\circ h_K} & (S^p)_{\Q}^n\\
	\left\vert\bigwedge W,0\right\vert\ar[ru]_{\widehat{g}}
		}\end{xy}\]
which commutes up to homotopy. Choose such a homotopy $H\colon\widehat{g}\circ(\vert\iota^\sharp\vert\circ h_K)\simeq k$ and consider the mapping cylinder of $\vert\iota^\sharp\circ h_K\vert$.

\begin{center}
\def\svgwidth{0.6\textwidth}
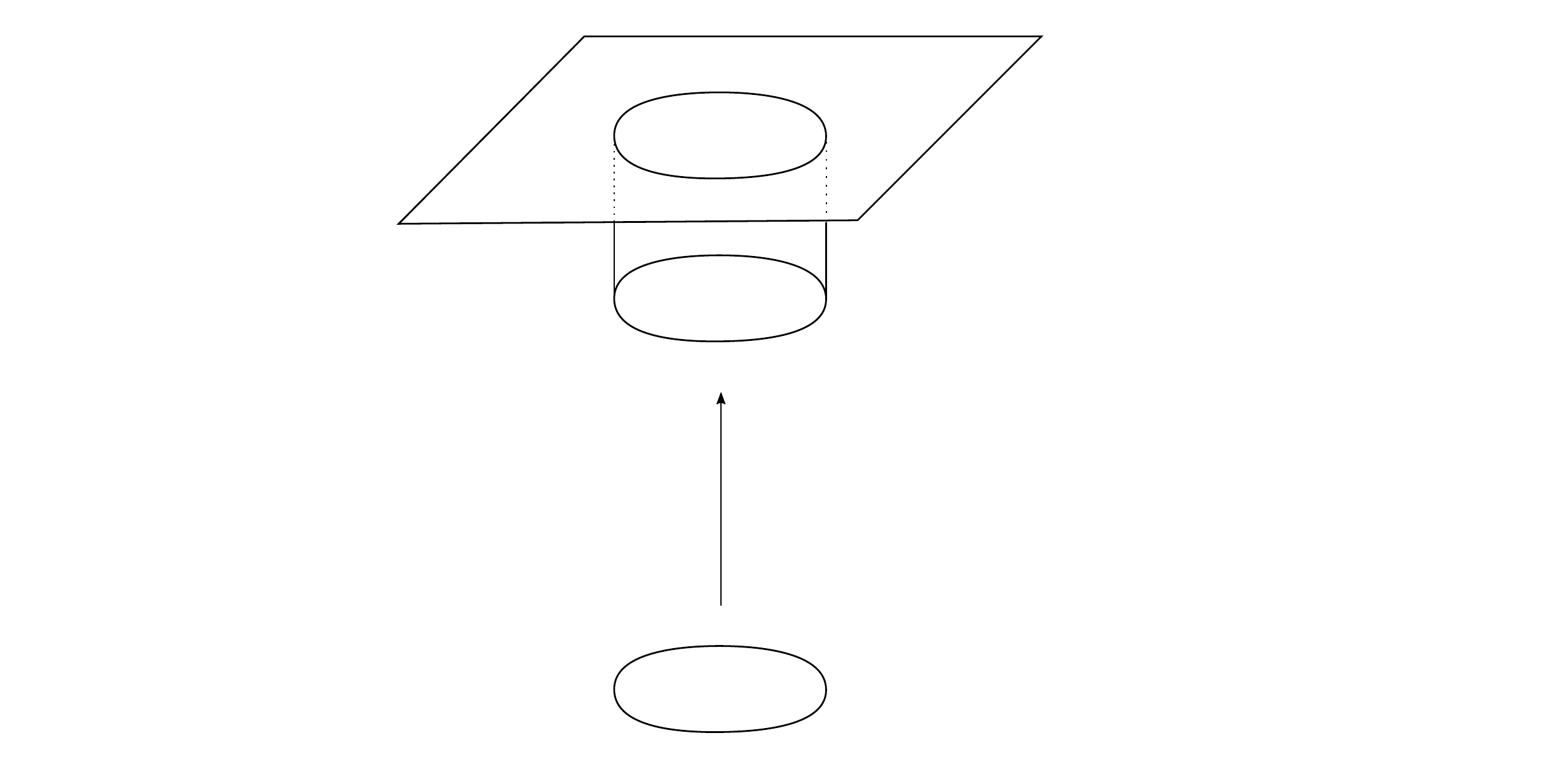
\end{center}

Using the homotopy $H$ we get a map $\overline{g}$ such that the diagram
\[\begin{xy}
  \xymatrix{
	M(\vert\iota^\sharp\vert\circ h_K)\ar[rd]^{\overline{g}} & \\
	K\ar@{^{(}->}[u]\ar[r]_k & (S^p)_{\Q}^n
}\end{xy}\]
commutes strictly. Choose a relative CW approximation $(\Fill(k),K)\to(M(\vert\iota^\sharp\vert\circ h_K),K)$, i.e. a relative CW complex $(\Fill(k),K)$ together with a map $\Fill(k)\to M(\vert\iota^\sharp\vert\circ h_K)$ which is a homotopy equivalence and restricts to the identity on $K$. Define $\iota$ and $\fil(k)$ as in the following diagram.
\[\begin{xy}\xymatrix{
\Fill(k)\ar@/^10mm/[rr]^{\fil(k)}\ar[r]^/-2mm/{\simeq} & M(\vert\iota^\sharp\vert\circ h_K)\ar[r] & (S^p)_{\Q}^n\\
 & K\ar@{^{(}->}[lu]^{\iota}\ar[u]\ar[ru]_k &
  }\end{xy}\]
The induced map $H^{p-1}(\iota;\Q)$ is surjective and $H^p(\iota;\Q)$ is injective since $\vert\iota^\sharp\vert\circ h_K$ has these properties. From this we get $H^p(\Fill(k),K;\Q)=0$ hence $H_p(\Fill(k),K;\Q)=0$. As usual we successively conclude that $H_p(\iota;\Q)$ is surjective and $\rk H^p(\fil(k);\Q)=\rk H^p(k;\Q)$.
\end{proof}
\begin{rem}
\begin{enumerate}[(i)]
\item In the case $p=1$ the factorisation (\ref{diag}) reminds us of our original Filling Lemma \ref{fill}.
\item The condition $n\leq p-2$ seems a little inorganic. But in the case $n=p-1$ the element $x$ could be of degree $np$ and therefore the product of $p$ generators of degree $(p-1)$ and we would not have any control over the image $H^{np}\iota^\sharp[x]$. We do not know how to weaken this condition. This may be possible by altering the construction of the Rational Filling Lemma.
\item If $p$ is even a minimal model of $(S^p)^n$ is given by $\left(\bigwedge[x_1,\ldots,x_n,y_1,\ldots,y_n],d\right)$ with $dy_i=x_i^2$. However it is not clear what the image of $y_i$ under the map $g^\sharp$ should be such that diagram (\ref{diag}) commutes or how to alter the construction.
\item It is remarkable that the Rational Filling Lemma can be proven while almost exclusively manipulating algebraic objects.
\end{enumerate}
\end{rem}
\begin{proof}[Proof of Theorem \ref{higherdeg}]
We will only indicate how to change the existing proof scheme. Again we proceed by contradiction and assume that $N$ connected and closed, there exists a smooth $f\colon M^{np}\to N^q$ and a triangulation $\mathcal{T}$ of $N$ such that the following two properties hold:
\begin{enumerate}[(i)]
\item The smooth simplices of $\mathcal{T}$ intersect $f$ stratum transversally.
\item For every $\sigma\in\mathcal{T}_k$ the inclusion $f_{\sigma}\colon F_{\sigma}\defeq f^{-1}\sigma(\Delta^k)\hookrightarrow M$ satisfies \[\rk H^p(f_{\sigma};\Q)<n-q\text{.}\]
\end{enumerate}
Again for every $\sigma\in\mathcal{T}_k$ we consider the simplices $z_{\sigma}\in\left(cl^{np-q}C_*(F_\sigma;\Q)\right)_k$ and the canonical cycle $Z(f;\mathcal{T})\in C_qcl^{np-q}C_*(M;\Q)$ from Construction \ref{canonical}. Let $r_M\colon M\to(S^p)_{\Q}^n$ be the rationalisation map of $M$. The diagram
\[\begin{xy}
  \xymatrix{
	Z(f;\mathcal{T})\ar@{|->}[rrr]^{\left(r_M\right)_*}\ar@{|->}[ddd]_{\ev_q} & & & \left(r_M\right)_*Z(f;\mathcal{T})\ar@{|->}[ddd]^{\ev_q}\\
	 & C_\bullet cl^{np-q}C_*(M;\Q)\ar[r]\ar[d] & C_\bullet cl^{np-q}C_*((S^p)_{\Q}^n;\Q)\ar[d] & \\
	 & C_{(np-q)+*}(M;\Q)\ar[r] & C_{(np-q)+*}((S^p)_{\Q}^n;\Q) & \\
	\widehat{Z(f;\mathcal{T})}\ar@{|->}[rrr]_{\left(r_M\right)_*} & & & \left(r_M\right)_*\widehat{Z(f;\mathcal{T})}
  }\end{xy}\]
commutes. The bottom left cycle represents the fundamental class $[M]_\Q\in H^{np}(M;\Q)$ and by definition $r_M$ is a rational homology equivalence, in particular the bottom right cycle defines a non-zero element in $H_{np}((S^p)_{\Q}^n;\Q)$. Therefore the top right cycle defines a non-zero element in $H_qcl^{np-q}C_*((S^p)_{\Q}^n;\Q)$.

Now we can use the Rational Filling Lemma \ref{ratfill}, proceed as earlier and deduce a contradiction by constructing a cone of $\left(r_M\right)_*Z(f;\mathcal{T})$ in $cl^{np-q}C_*((S^p)_{\Q}^n;\Q)$ via simplices $w_{\sigma}\in\left(cl^{np-q}C_*((S^p)_{\Q}^n;\Q)\right)_{q+1}$ satisfying \[\del_iw_{\sigma}=\begin{cases}w_{\del_i\sigma}&\text{if }0\leq i\leq k\\z_{\sigma}&\text{if }i={k+1}\text{.}\end{cases}\] For $\sigma\in\mathcal{T}_0$ and $i=0$ the equation above shall be interpreted as $\del_0w_{\sigma}=w_{\del_0\sigma}=0$.
\end{proof}
\bibliographystyle{abbrv}
\bibliography{algfill}
\end{document}

%% file: cycles.pdf_tex
\begingroup%
  \makeatletter%
  \providecommand\color[2][]{%
    \errmessage{(Inkscape) Color is used for the text in Inkscape, but the package 'color.sty' is not loaded}%
    \renewcommand\color[2][]{}%
  }%
  \providecommand\transparent[1]{%
    \errmessage{(Inkscape) Transparency is used (non-zero) for the text in Inkscape, but the package 'transparent.sty' is not loaded}%
    \renewcommand\transparent[1]{}%
  }%
  \providecommand\rotatebox[2]{#2}%
  \ifx\svgwidth\undefined%
    \setlength{\unitlength}{569.071661bp}%
    \ifx\svgscale\undefined%
      \relax%
    \else%
      \setlength{\unitlength}{\unitlength * \real{\svgscale}}%
    \fi%
  \else%
    \setlength{\unitlength}{\svgwidth}%
  \fi%
  \global\let\svgwidth\undefined%
  \global\let\svgscale\undefined%
  \makeatother%
  \begin{picture}(1,0.66104986)%
    \put(0,0){\includegraphics[width=\unitlength,page=1]{cycles.pdf}}%
    \put(0.18072871,0.29536738){\color[rgb]{0,0,0}\makebox(0,0)[lb]{\smash{$f$}}}%
    \put(0.11256531,0.56447081){\color[rgb]{0,0,0}\makebox(0,0)[lb]{\smash{$\dim f^{-1}[v,w]=2$}}}%
    \put(0,0){\includegraphics[width=\unitlength,page=2]{cycles.pdf}}%
    \put(0.13386638,0.16472089){\color[rgb]{0,0,0}\makebox(0,0)[lb]{\smash{$[v,w]$}}}%
    \put(0.04087264,0.19052583){\color[rgb]{0,0,0}\makebox(0,0)[lb]{\smash{$v$}}}%
    \put(0.25883261,0.20022261){\color[rgb]{0,0,0}\makebox(0,0)[lb]{\smash{$w$}}}%
    \put(-0.00175038,0.38270174){\color[rgb]{0,0,0}\makebox(0,0)[lb]{\smash{$c_v$}}}%
    \put(0,0){\includegraphics[width=\unitlength,page=3]{cycles.pdf}}%
    \put(0.2709032,0.39406225){\color[rgb]{0,0,0}\makebox(0,0)[lb]{\smash{$c_w$}}}%
    \put(0,0){\includegraphics[width=\unitlength,page=4]{cycles.pdf}}%
    \put(0.0380116,0.49133712){\color[rgb]{0,0,0}\makebox(0,0)[lb]{\smash{$c_{[v,w]}$}}}%
    \put(0,0){\includegraphics[width=\unitlength,page=5]{cycles.pdf}}%
    \put(0.06712305,0.10294779){\color[rgb]{0,0,0}\makebox(0,0)[lb]{\smash{$n-q=1$, $k=1$}}}%
    \put(0,0){\includegraphics[width=\unitlength,page=6]{cycles.pdf}}%
    \put(0.74844438,0.25240927){\color[rgb]{0,0,0}\makebox(0,0)[lb]{\smash{$f$}}}%
    \put(0.66510059,0.6450287){\color[rgb]{0,0,0}\makebox(0,0)[lb]{\smash{$\dim f^{-1}[u,v,w]=3$}}}%
    \put(0,0){\includegraphics[width=\unitlength,page=7]{cycles.pdf}}%
    \put(0.56970717,0.07467617){\color[rgb]{0,0,0}\makebox(0,0)[lb]{\smash{$u$}}}%
    \put(0.86793724,0.06764718){\color[rgb]{0,0,0}\makebox(0,0)[lb]{\smash{$v$}}}%
    \put(0.73639471,0.21425194){\color[rgb]{0,0,0}\makebox(0,0)[lb]{\smash{$w$}}}%
    \put(0.60083555,0.00438626){\color[rgb]{0,0,0}\makebox(0,0)[lb]{\smash{$n-q=1$, $k=2$}}}%
    \put(0,0){\includegraphics[width=\unitlength,page=8]{cycles.pdf}}%
    \put(0.89705736,0.14597024){\color[rgb]{0,0,0}\makebox(0,0)[lb]{\smash{$[u,v,w]$}}}%
    \put(0.50142552,0.26947972){\color[rgb]{0,0,0}\makebox(0,0)[lb]{\smash{$c_u$}}}%
    \put(0,0){\includegraphics[width=\unitlength,page=9]{cycles.pdf}}%
    \put(0.53556632,0.51850685){\color[rgb]{0,0,0}\makebox(0,0)[lb]{\smash{$c_{[u,w]}$}}}%
    \put(0.8528751,0.55465592){\color[rgb]{0,0,0}\makebox(0,0)[lb]{\smash{$c_{[u,v,w]}$}}}%
    \put(0,0){\includegraphics[width=\unitlength,page=10]{cycles.pdf}}%
  \end{picture}%
\endgroup%

%% file: torus.pdf_tex
\begingroup%
  \makeatletter%
  \providecommand\color[2][]{%
    \errmessage{(Inkscape) Color is used for the text in Inkscape, but the package 'color.sty' is not loaded}%
    \renewcommand\color[2][]{}%
  }%
  \providecommand\transparent[1]{%
    \errmessage{(Inkscape) Transparency is used (non-zero) for the text in Inkscape, but the package 'transparent.sty' is not loaded}%
    \renewcommand\transparent[1]{}%
  }%
  \providecommand\rotatebox[2]{#2}%
  \ifx\svgwidth\undefined%
    \setlength{\unitlength}{443.94840393bp}%
    \ifx\svgscale\undefined%
      \relax%
    \else%
      \setlength{\unitlength}{\unitlength * \real{\svgscale}}%
    \fi%
  \else%
    \setlength{\unitlength}{\svgwidth}%
  \fi%
  \global\let\svgwidth\undefined%
  \global\let\svgscale\undefined%
  \makeatother%
  \begin{picture}(1,0.44859879)%
    \put(0,0){\includegraphics[width=\unitlength,page=1]{torus.pdf}}%
    \put(0.42895184,0.16066306){\color[rgb]{0,0,0}\makebox(0,0)[lb]{\smash{$f$}}}%
    \put(-0.00224372,0.40650894){\color[rgb]{0,0,0}\makebox(0,0)[lb]{\smash{$T^2$}}}%
    \put(0.90004909,0.03580947){\color[rgb]{0,0,0}\makebox(0,0)[lb]{\smash{$\R^1$}}}%
    \put(0,0){\includegraphics[width=\unitlength,page=2]{torus.pdf}}%
    \put(0.42380323,0.0422452){\color[rgb]{0,0,0}\makebox(0,0)[lb]{\smash{$y$}}}%
    \put(0,0){\includegraphics[width=\unitlength,page=3]{torus.pdf}}%
    \put(0.58984568,0.15036587){\color[rgb]{0,0,0}\makebox(0,0)[lb]{\smash{$\Phi_f$}}}%
    \put(0,0){\includegraphics[width=\unitlength,page=4]{torus.pdf}}%
  \end{picture}%
\endgroup%

%% file: kegel.pdf_tex
\begingroup%
  \makeatletter%
  \providecommand\color[2][]{%
    \errmessage{(Inkscape) Color is used for the text in Inkscape, but the package 'color.sty' is not loaded}%
    \renewcommand\color[2][]{}%
  }%
  \providecommand\transparent[1]{%
    \errmessage{(Inkscape) Transparency is used (non-zero) for the text in Inkscape, but the package 'transparent.sty' is not loaded}%
    \renewcommand\transparent[1]{}%
  }%
  \providecommand\rotatebox[2]{#2}%
  \ifx\svgwidth\undefined%
    \setlength{\unitlength}{342.29726563bp}%
    \ifx\svgscale\undefined%
      \relax%
    \else%
      \setlength{\unitlength}{\unitlength * \real{\svgscale}}%
    \fi%
  \else%
    \setlength{\unitlength}{\svgwidth}%
  \fi%
  \global\let\svgwidth\undefined%
  \global\let\svgscale\undefined%
  \makeatother%
  \begin{picture}(1,0.58725319)%
    \put(0,0){\includegraphics[width=\unitlength,page=1]{kegel.pdf}}%
    \put(0.2936649,0.0082679){\color[rgb]{0,0,0}\makebox(0,0)[lb]{\smash{$z_{\del_i\sigma}$}}}%
    \put(0,0){\includegraphics[width=\unitlength,page=2]{kegel.pdf}}%
    \put(0.60708608,0.04527643){\color[rgb]{0,0,0}\makebox(0,0)[lb]{\smash{$z_{\sigma}$}}}%
    \put(0.48321714,0.52906638){\color[rgb]{0,0,0}\makebox(0,0)[lb]{\smash{$0$}}}%
    \put(0,0){\includegraphics[width=\unitlength,page=3]{kegel.pdf}}%
    \put(0.26118791,0.28366564){\color[rgb]{0,0,0}\makebox(0,0)[lb]{\smash{$w_{\del_i\sigma}$}}}%
    \put(0,0){\includegraphics[width=\unitlength,page=4]{kegel.pdf}}%
    \put(0.74965222,0.36897161){\color[rgb]{0,0,0}\makebox(0,0)[lb]{\smash{$w_{\sigma}$}}}%
    \put(-0.00291003,0.5606179){\color[rgb]{0,0,0}\makebox(0,0)[lb]{\smash{$cl^{n-q}C_*(T^n;\Z_2)$}}}%
  \end{picture}%
\endgroup%

%% file: glue.pdf_tex
\begingroup%
  \makeatletter%
  \providecommand\color[2][]{%
    \errmessage{(Inkscape) Color is used for the text in Inkscape, but the package 'color.sty' is not loaded}%
    \renewcommand\color[2][]{}%
  }%
  \providecommand\transparent[1]{%
    \errmessage{(Inkscape) Transparency is used (non-zero) for the text in Inkscape, but the package 'transparent.sty' is not loaded}%
    \renewcommand\transparent[1]{}%
  }%
  \providecommand\rotatebox[2]{#2}%
  \ifx\svgwidth\undefined%
    \setlength{\unitlength}{384.15805664bp}%
    \ifx\svgscale\undefined%
      \relax%
    \else%
      \setlength{\unitlength}{\unitlength * \real{\svgscale}}%
    \fi%
  \else%
    \setlength{\unitlength}{\svgwidth}%
  \fi%
  \global\let\svgwidth\undefined%
  \global\let\svgscale\undefined%
  \makeatother%
  \begin{picture}(1,0.60263129)%
    \put(0,0){\includegraphics[width=\unitlength,page=1]{glue.pdf}}%
    \put(0.61988932,0.16625391){\color[rgb]{0,0,0}\makebox(0,0)[lb]{\smash{$f$}}}%
    \put(0,0){\includegraphics[width=\unitlength,page=2]{glue.pdf}}%
    \put(0.53433188,0.00649757){\color[rgb]{0,0,0}\makebox(0,0)[lb]{\smash{$\sigma=[v,w]$}}}%
    \put(0.40676391,0.06152151){\color[rgb]{0,0,0}\makebox(0,0)[lb]{\smash{$v$}}}%
    \put(0.74153826,0.05803606){\color[rgb]{0,0,0}\makebox(0,0)[lb]{\smash{$w$}}}%
    \put(0,0){\includegraphics[width=\unitlength,page=3]{glue.pdf}}%
    \put(-0.00208252,0.42265738){\color[rgb]{0,0,0}\makebox(0,0)[lb]{\smash{$n-q=1$, $k=1$}}}%
    \put(0,0){\includegraphics[width=\unitlength,page=4]{glue.pdf}}%
    \put(0.29315879,0.25494115){\color[rgb]{0,0,0}\makebox(0,0)[lb]{\smash{$L_v$}}}%
    \put(0.90425999,0.35381123){\color[rgb]{0,0,0}\makebox(0,0)[lb]{\smash{$L_w$}}}%
    \put(0,0){\includegraphics[width=\unitlength,page=5]{glue.pdf}}%
    \put(0.58871718,0.57889838){\color[rgb]{0,0,0}\makebox(0,0)[lb]{\smash{$F_{\sigma}$}}}%
    \put(-0.00259293,0.49643975){\color[rgb]{0,0,0}\makebox(0,0)[lb]{\smash{$K_{\sigma}=K_{\sigma}^{(0)}$}}}%
  \end{picture}%
\endgroup%

%% file: cylinder.pdf_tex
\begingroup%
  \makeatletter%
  \providecommand\color[2][]{%
    \errmessage{(Inkscape) Color is used for the text in Inkscape, but the package 'color.sty' is not loaded}%
    \renewcommand\color[2][]{}%
  }%
  \providecommand\transparent[1]{%
    \errmessage{(Inkscape) Transparency is used (non-zero) for the text in Inkscape, but the package 'transparent.sty' is not loaded}%
    \renewcommand\transparent[1]{}%
  }%
  \providecommand\rotatebox[2]{#2}%
  \ifx\svgwidth\undefined%
    \setlength{\unitlength}{723.25752563bp}%
    \ifx\svgscale\undefined%
      \relax%
    \else%
      \setlength{\unitlength}{\unitlength * \real{\svgscale}}%
    \fi%
  \else%
    \setlength{\unitlength}{\svgwidth}%
  \fi%
  \global\let\svgwidth\undefined%
  \global\let\svgscale\undefined%
  \makeatother%
  \begin{picture}(1,0.48653324)%
    \put(0,0){\includegraphics[width=\unitlength,page=1]{cylinder.pdf}}%
    \put(-0.00137723,0.33111021){\color[rgb]{0,0,0}\makebox(0,0)[lb]{\smash{$M(\vert\iota^\sharp\vert\circ h_K)$}}}%
    \put(0.68598904,0.43540027){\color[rgb]{0,0,0}\makebox(0,0)[lb]{\smash{$\left\vert\bigwedge W,0\right\vert$}}}%
    \put(0.33203494,0.03325152){\color[rgb]{0,0,0}\makebox(0,0)[lb]{\smash{$K$}}}%
    \put(0,0){\includegraphics[width=\unitlength,page=2]{cylinder.pdf}}%
    \put(0.85980584,0.04352244){\color[rgb]{0,0,0}\makebox(0,0)[lb]{\smash{$(S^p)_{\Q}^n$}}}%
    \put(0,0){\includegraphics[width=\unitlength,page=3]{cylinder.pdf}}%
    \put(0.71127152,0.25210261){\color[rgb]{0,0,0}\makebox(0,0)[lb]{\smash{$\overline{g}$}}}%
    \put(0.65122573,0.00243851){\color[rgb]{0,0,0}\makebox(0,0)[lb]{\smash{$k$}}}%
  \end{picture}%
\endgroup%